\documentclass[10pt,a4paper]{article}
\usepackage[utf8]{inputenc}
\usepackage[english]{babel} 
\usepackage[francais]{layout}
\usepackage[T1]{fontenc}
\usepackage{inputenc}
\usepackage{eurosym}
\usepackage{aeguill}
\usepackage{a4wide} 
\usepackage{amsmath}
\numberwithin{equation}{section}
\usepackage{ntheorem}
\theorembodyfont{\upshape}
\newtheorem{theorem}{Theorem}[section]
\newtheorem{lemma}{Lemma}[section]
\newtheorem{definition}{Definition}[section]
\newtheorem{remark}{Remark}[section]

\newtheorem*{proof}{Proof.}

\newtheorem{proposition}{Proposition}[section]

\newtheorem{corollary}{Corollary}[section]

\usepackage{makeidx} \makeindex
\usepackage[Lenny]{fncychap}
\usepackage[mathscr]{eucal}			
\usepackage{fancyhdr}
\usepackage{fancybox}
\usepackage{shadethm}
\usepackage{lastpage}
\usepackage{amsfonts}
\usepackage{amsmath}
\usepackage{amssymb}
\usepackage{latexsym}
\usepackage{lettrine}
\usepackage[hmargin=3cm,vmargin=2cm]{geometry}
\usepackage{float}
\usepackage{graphicx}
\usepackage{color,framed}
\usepackage{graphicx}
\usepackage{graphics,wrapfig}
\usepackage{hyperref,graphicx}
\usepackage[all,line]{xy}

\begin{document}
\global\long\def\b{\mathcal{B}}

\global\long\def\ro{\rho}

\global\long\def\e{\mathcal{E}}

\global\long\def\m{\mathcal{M}}

\global\long\def\vers{\longrightarrow}

\global\long\def\too{\longrightarrow}

\global\long\def\verss{\longmapsto}

\global\long\def\fii{\phi}

\global\long\def\ph{\varphi}

\global\long\def\te{\theta}

\global\long\def\sig{\sigma}

\global\long\def\infi{\infty}

\global\long\def\ds{\in}

\global\long\def\norm#1{\left|#1\right|}

\global\long\def\norme#1{\|#1\|}

\global\long\def\app{]-\infi,0]\vers[0,+\infi[}

\global\long\def\fase#1{]-\infi,#1]}

\global\long\def\teta{\theta}

\global\long\def\om{\Omega}

\global\long\def\sig{\sigma}

\global\long\def\eps{\varepsilon}

\global\long\def\lda{\lambda}

\global\long\def\da{\lambda}

\global\long\def\C{\mathbb{C}}

\global\long\def\R{\mathbb{R}}

\global\long\def\F{\mathbb{F}}

\global\long\def\X{\mathbb{X}}

\global\long\def\x{\mathbb{X}}

\global\long\def\f{\mathbb{F}}

\global\long\def\n{\mathbb{N}}

\global\long\def\N{\mathbb{N}}

\global\long\def\r{\mathbb{R}}

\global\long\def\fb{\mathbb{F}_{b}}

\global\long\def\fbf{\mathbb{F}_{b}\left(\fii\right)}

\title{\textbf{Compact almost automorphic solutions for  semilinear parabolic  evolution equations}}

\author{Brahim Es-sebbar$^1$ 
  \and
Khalil Ezzinbi$^2$
  \and
 Kamal Khalil$^{2,}$\footnote{Corresponding author: \texttt{kamal.khalil@ced.uca.ma}}
}

\newcommand{\Addresses}{{
  \bigskip
  \footnotesize

  Brahim Es-Sebbar, \textsc{Cadi Ayyad University, Faculty of Sciences and Technology Gueliz, Marrakech, Morocco
}\par\nopagebreak
  \textit{E-mail address}, B. Es-sebbar: \texttt{essebbar@live.fr }

  \medskip

  Khalil Ezzinbi, \textsc{Department of Mathematics, Faculty of Sciences Semlalia, Cadi Ayyad University, Marrakesh B.P. 2390-40000, Morocco.}\par\nopagebreak
  \textit{E-mail address}, K. Ezzinbi: \texttt{ezzinbi@uca.ac.ma}

  \medskip

  Kamal Khalil, \textit{\textsc{Department of Mathematics, Faculty of Sciences Semlalia, Cadi Ayyad University, Marrakesh B.P. 2390-40000, Morocco.}}\par\nopagebreak
  \textit{E-mail address}, K. Khalil (Corresponding author): \texttt{kamal.khalil@ced.uca.ma}
  \date{{\footnotesize \today}}
 
}}
\newcommand{\Class}{{
 \bigskip
  \footnotesize
\emph{2010 Mathematics subject classification.} Primary 46T20-47J35; Secondary 34C27-35K58.}}
\maketitle
\begin{center}
{\small $^1$ Cadi Ayyad University, Faculty of Sciences and Technology Gueliz, Marrakech, Morocco.}\\
{\small $^2$ Department of Mathematics, Faculty of Sciences Semlalia, Cadi Ayyad University, Marrakesh B.P. 2390-40000, Morocco.} 
\end{center}



\vspace*{+0.2cm}
\textbf{\abstractname{.}} 
In this paper, using the subvariant functional method due to Favard \cite{Favard}, we prove the existence of a unique compact almost automorphic solution for a class of semilinear evolution equations in Banach spaces provided the existence of at least one bounded solution on the right half line. More specifically, we improve the assumptions in \cite{CieuEzz}, we show that the almost automorphy of the coefficients in a weaker sense (Stepanov almost automorphy of order $1\leq p <\infty$) is enough to obtain solutions that are almost automorphic in a strong sense (Bochner almost automorphy). For that purpose we distinguish two cases, $ p=1 $ and $ p>1$. The main difficulty in this work, is to prove the existence of at least one solution with relatively compact range while the forcing term is not necessarily bounded. Moreover, we propose to study a large class of reaction-diffusion problems with unbounded forcing terms. \\


\section{Introduction}\label{section1Pep3}

The concept of almost automorphy introduced by Bochner \cite{Boch2} is not restricted just to continuous functions. One can generalize that notion to measurable functions with some suitable conditions of integrability, namely, Stepanov almost automorphic functions, see \cite{Moi2,Boch2,EssEzz1}. That is a Stepanov almost automorphic function is neither continuous nor bounded necessarily. 

Now, consider the following semilinear evolution equation in a Banach space $X$:
\begin{equation}
  x'(t)= Ax(t) + f(t, x(t))\quad \text{for}\; t \in \mathbb{R},  
   \label{Eq1SubStep}
 \end{equation}
where $ (A,D(A)) $ is generator of a $ C_{0} $-semigroup $ (T(t))_{t\geq 0} $ on $X$. The function $ f : \mathbb{R}\times X \longrightarrow X $ is locally integrable of order $ 1\leq p < \infty $ in $ t $ and continuous in $ x $.\\
The study of existence of almost periodic and almost automorphic solutions to equation \eqref{Eq1SubStep} in infinite dimensional Banach spaces was deeply investigated in the last decades, we refer to \cite{Moi,Moi2,CieuEzz,Diag1,Het,Langa,CZhang,Zheng,Zhenxin}. Recently, in \cite{Moi}, the authors studied equation  \eqref{Eq1SubStep}, in the parabolic context, that is when $(A,D(A))$ generates an analytic semigroup $ (T(t))_{t\geq 0} $ on a Banach space $X$ which has an exponential dichotomy on $\R$ and $f$ is Stepanov almost periodic of order  $ 1\leq p < \infty $ and Lipschitzian with respect to $x$. They proved the existence and uniqueness of almost periodic solutions for equation \eqref{Eq1SubStep}. In \cite{Diag1}, the authors proved the existence and uniqueness of almost automorphic solutions for equation \eqref{Eq1SubStep} in the case where  $(A,D(A))$ generates an exponentially stable $ C_0 $-semigroup $ (T(t))_{t\geq 0} $ on a Banach space $X$ and $f$ is $S^p$-almost automorphic in $t$ (of order $1<p<\infty$) and Lipschitzian with respect to $x$. Moreover, in \cite{CieuEzz}, the authors proved the existence of compact almost automorphic solutions for equation \eqref{Eq1SubStep} provided that $ (A,D(A)) $ generates a compact $ C_{0} $-semigroup and $f$ is almost automorphic in the classical sense uniformly with respect to $x$. The results are obtained using the subvariant functional method introduced in \cite{Fink}.\\

In this paper, using the subvariant functional method, we give sufficient conditions insuring the existence and uniqueness of compact almost automorphic solutions to equation \eqref{Eq1SubStep}. We assume that the $ C_0 $-semigroup $ (T(t))_{t\geq 0} $ is compact and $f$ is just Stepanov almost automorphic of order $1\leq p  <\infty$. Actually, we begin by introducing a very useful characterization of uniformly  Stepanov almost automorphic functions (Lemma \ref{LemApSpCom1Paper3}). That is a uniformly $S^p $-almost automorphic function is pointwise  $S^p $-almost automorphic and uniformly continuous on compact sets of $X$ in a weak sense. Then we prove that a mild solution $x$ of equation \eqref{Eq1SubStep} given by: 
$$
x(t)=T(t-s)x(s)+\int_{s}^{t}T(t-\sigma)f(\sigma,x(\sigma
))\ d\sigma\qquad\text{ for }t\geq s, \, t,s \in \R 
$$
with relatively compact range is uniformly continuous, see Theorem \ref{Lemma31}. Technically, we distinguish two cases, $ 1<p<\infty $ and $p=1$. Therefore, using Ascoli-Arzela Theorem and an extraction diagonal argument, we construct a mild solution $x$ on $\mathbb{R}$ such that its range is relatively compact, provided  that equation \eqref{Eq1SubStep} has at least a  mild solution $x_0 $ on the right half line with relatively compact range satisfying the following inclusion: 
\begin{eqnarray*}
\lbrace x(t):\; t \in \mathbb{R} \rbrace \subset \overline{ \lbrace x_{0}(t):\; t\geq t_{0} \rbrace },
\end{eqnarray*}
see Lemma \ref{Propostion1}. Note that, the assumption of existence of mild solutions on the right half line with relatively compact range of equation \eqref{Eq1SubStep} is strongly needed  in our work. Consequently, sufficient conditions insuring that assumption are improved in Lemmas \ref{Lemma Relativ Comp p not 1}, \ref{Propostion12} and Proposition \ref{Proposition H5}.  Here, the situation is quite difficult since $f$ is not necessarily bounded. Indeed, we consider two cases, $ 1<p<\infty $ and $p=1$. For $ 1<p<\infty $, in Lemma \ref{Lemma Relativ Comp p not 1}, we prove that equation \eqref{Eq1SubStep} has a mild solution on $\R ^+ $ with relatively compact range provided it has at least a bounded mild solution on $\R ^+ $. In general, this result does not hold for $p=1$ without any additional assumptions on  $f$, see Lemma \ref{Propostion12}. Furthermore, when $f$ satisfies hypothesis \textbf{(H5)} below, the result also holds in general, see Proposition \ref{Proposition H5}. After proving that, we give our main results, we show that,  if equation \eqref{Eq1SubStep} has at least a  bounded mild solution $x_{0}$ on the right half line, then it has a unique compact almost automorphic solution provided the existence of a unique mild solution of equation \eqref{Eq1SubStep} which minimizes a given subvariant functional. Here, we differentiate two cases, $ 1<p<\infty $, see Theorem \ref{Main Theorem 1} and the case $p=1$, see Theorem \ref{Main Theorem 2}. 

Furthermore, to illustrate our theoretical results, we study a class of reaction diffusion problems of the form: 
\begin{equation}
\left\{
\begin{array}
[c]{l}%
\dfrac{\partial}{\partial t}v(t,\xi)=\displaystyle \sum_{k=1}^{n}\dfrac{\partial^{2}}{\partial \xi_{i}^{2}%
}v(t,\xi)+g(v(t,\xi))+h\left(  t,\xi\right)  \text{ for }t\in\mathbb{R}\text{ and
}\xi\in\Omega ,\\
\text{ \ \ \ \ \ \ \ \ \ \ \ \ \ \ \ \ \ \ \ \ \ }\\
v(t,\xi)=0\; \text{ for }t\in\mathbb{R}, \ \xi\in \partial \Omega ,
\end{array}
\right.  \label{AppPap3Intr}
\end{equation}
in a bounded open  domain $ \Omega \subset \r^{n}, \ n\geq 1$ with smooth boundary $ \partial \Omega $ and $h:\mathbb{R}\times\Omega \rightarrow\mathbb{R}$ is not necessarily bounded. To satisfy our assumptions, we introduce the main fact that, equation \eqref{AppPap3Intr} has a unique bounded global solution on $\R ^+ $,  see Theorem \ref{Proposition Application}. Our result is of independent interest and it generalizes many works in the literature, see \cite{Haraux,Roth}. Therefore, we give our main Theorem of the existence of a unique compact almost automorphic solution to equation \eqref{AppPap3Intr},  see Theorem \ref{Theorem Application main}.\\

The work is organized as follows: Section \ref{section2Pep3} is devoted to preliminaries for Stepanov almost automorphic functions and the subvariant minimizing functional. In Section \ref{section3}, we give our main results of this work. We begin by the characterization of uniformly Stepanov almost automorphic functions. After that, we prove under the weak assumption of Stepanov almost automorphy on $f$ that every mild solution of \eqref{Eq1SubStep} with relatively compact range is uniformly continuous. In Section \ref{section31Pep3} we give sufficient conditions to the existence of solutions with relatively compact ranges to equation \eqref{Eq1SubStep} on the right half line. Section \ref{section32Pep3} is devoted to main Theorems \ref{Main Theorem 1} and \ref{Main Theorem 2}, we prove the existence of compact almost automorphic solutions through a minimizing some subvariant functional provided the existence of a bounded mild solution on the right half line. Finally, Section \ref{section4} concerns the application \eqref{AppPap3Intr}.

\section{Preliminaries}\label{section2Pep3}
Throughout this work, $(X,\|\cdot\|) $ is a Banach space. $BC(\mathbb{R},X)$ equipped with the supremum norm, the Banach space of bounded continuous functions $f$ from $\mathbb{R}$ into $ X$. For $ 1\leq p <\infty $, $ q $ denotes its conjugate exponent defined by $ \dfrac{1}{p} +\dfrac{1}{q}=1$ if $ p\neq  1$ and $ q=\infty $ if $ p=1 $. By  $ L^{p}_{loc}(\mathbb{R},X)$ (resp. $ L^{p}(\mathbb{R},X)$), we designate the space (resp. the Banach space) of all equivalence classes of measurable functions $f$ from $\mathbb{R}$ into $ X$  such that $\|f(\cdot)\|^{p}$ is locally integrable (resp. integrable).


\subsection{Almost automorphic functions}\label{section21Pep3}
\begin{definition}[H. Bohr]\cite{H.Bohr}
A continuous function $ f:\mathbb{R}\longrightarrow X $ is said to be almost periodic if for every $\varepsilon > 0,$ there exists $ l_{\varepsilon}  > 0 $, such that for every $ a\in \mathbb{R} $, there exists $ \tau \in \left[a,a+ l_{\varepsilon} \right]  $ satisfying: $$ \| f(t+\tau)-f(t) \| < \varepsilon \quad \mbox{for all}\; t\in \mathbb{R} .$$ The space of all such functions is denoted by $ AP(\mathbb{R},X) .$
\end{definition}
\begin{definition}[S. Bochner]\cite{Boch2}\label{DefAA}
A continuous function $ f:\mathbb{R}\longrightarrow X $ is called almost automorphic if for every sequence $(\sigma_{n})_{n \geq 0}$ of real numbers, there exist a subsequence $(s_{n})_{n \geq 0}\subset (\sigma_{n})_{n \geq 0} $ and $g:\mathbb{R}\longrightarrow X $ such that,  for each $t\in \r$
\begin{eqnarray*}
g(t)=:\lim_{n} f(t+s_{n})\quad \mbox{and}\quad f(t)=\lim_{n} g(t-s_{n}) .\label{DefAlmAuom}
\end{eqnarray*} 
If the above limits hold uniformly on compact subsets
of $\r$, then $f$ is called compact almost automorphic.
\end{definition}
\noindent In the sequel, $AA(\mathbb{R},X)$ (resp. $KAA(\mathbb{R},X)$)
denotes the space of almost automorphic (resp. compact almost automorphic)
$X$-valued functions.

\begin{proposition}\cite{Ess2}\label{EssebbarCharac}
Let $f \in AA(\mathbb{R},X)  $. Then, $f$ is compact almost automorphic if and only if it is uniformly continuous.
\end{proposition}
\begin{remark}
\textbf{(i)} By the pointwise convergence, the function $g$ in Definition
\ref{DefAA} is only measurable and bounded but not necessarily
continuous. If one of the two convergences in Definition \ref{DefAA}
is uniform on $\r$, then $f$ is almost periodic. For more details
about this topic we refer the reader to the book \cite{N'Gue}.\\
\textbf{(ii)} An almost automorphic function may not be uniformly continuous. In fact, the real function $$ f(t)= \sin\left( \dfrac{1}{2+\cos(t)+\cos(\sqrt{2}t)}\right)\; \text{for } \;  t\in \mathbb{R}, $$ belong to $  AA(\mathbb{R},\mathbb{R}) $, but is not uniformly continuous.
\end{remark}
Then, we have the following inclusions:
\begin{eqnarray}
AP(\mathbb{R},X) \subset KAA(\mathbb{R},X) \subset AA(\mathbb{R},X) \subset BC(\mathbb{R},X).
\end{eqnarray}
\definition\textbf{\cite{EssEzz1}} Let $ 1\leq p< \infty $.  A function $ f\in L^{p}_{loc} (\mathbb{R},X) $ is said to be bounded in the sense of Stepanov if $$ \displaystyle \sup_{t\in \mathbb{R}}  \left( \int_{\left[ t,t+1\right] } \|f(s)\|^{p}ds\right) ^{\frac{1}{p}}=\displaystyle \sup_{t\in \mathbb{R}}  \left( \int_{\left[ 0,1\right] } \|f(t+s)\|^{p}ds\right) ^{\frac{1}{p}} <\infty.$$ 
The space of all such functions is denoted by $ BS^{p} (\mathbb{R},X)$ and is provided with the following norm:
\begin{eqnarray*} 
   \|f\|_{BS^{p} }&:=& \displaystyle \sup_{t\in \mathbb{R}}  \left( \int_{\left[ t,t+1\right] } \|f(s)\|^{p}ds\right) ^{\frac{1}{p}} \\ & =& \sup_{t\in \mathbb{R}}\|f(t+\cdot)\|_{L^{p} (\left[ 0,1\right],X)} .
\end{eqnarray*}   \\ 
Then, the following inclusions hold:
\begin{eqnarray}
BC(\mathbb{R},X) \subset BS^p (\mathbb{R},X) \subset L^{p}_{loc} (\mathbb{R},X) . \label{incls}
\end{eqnarray} 
\begin{definition}[Bochner transform]\cite{Ess}
Let $f\in L^{p}_{loc} (\mathbb{R},X) $ for $ 1\leq p< \infty $. The Bochner transform of $f$ is the function $f^b :\mathbb{R}\longrightarrow L^{p} (\left[ 0,1\right],X) $ defined for all $t\in \mathbb{R} $ by $$ (f^b (t))(s)= f(t+s) \quad \mbox{for} \; s\in \left[ 0,1\right] .$$
\end{definition}

 Now, we give the definition of almost automorphy in the sense of Stepanov.
\begin{definition}\label{DefAlmAut1}\cite{EssEzz1} Let $ 1\leq p< \infty $. A function $ f \in L^{p}_{loc} (\mathbb{R},X) $ is said to be almost automorphic in the sense of Stepanov (or $S^p$-almost automorphic), if for every sequence $(\sigma_{n})_{n \geq 0}$ of real numbers, there exists a subsequence $(s_{n})_{n \geq 0}\subset (\sigma_{n})_{n \geq 0} $ and $g\in L^{p}_{loc} (\mathbb{R},X) $, such that, for each $ t\in \mathbb{R} $
\begin{eqnarray*}
 \lim_{n} \left( \int_{t}^{t+1}\| f(s+s_{n})-g(s)\|^{p}ds\right)^{\frac{1}{p}} =0\quad \mbox{and}\quad \lim_{n} \left( \int_{t}^{t+1}\| g(s-s_{n})-f(s)\|^{p}ds\right)^{\frac{1}{p}}=0.\label{DefAlmAuom5Pap3}
\end{eqnarray*}
 The space of all such functions is denoted by $ AAS^{p}(\mathbb{R},X) .$ 
\end{definition}
\begin{theorem}\cite{EssEzz1}
The following are equivalent:
\begin{itemize}
\item[\textbf{(i)}] $ f $ is $S^p$-almost automorphic.
\item[\textbf{(ii)}] For every sequence $(\sigma_{n})_{n \geq 0}$ of real numbers, there exists a subsequence $(s_{n})_{n \geq 0}\subset (\sigma_{n})_{n \geq 0} $ for each $ t\in \mathbb{R} $

\begin{eqnarray*}
\lim_{n,m} \left( \int_{t}^{t+1}\| f(\tau+s_{n}-s_{m}) -f(\tau)\|^{p}d\tau\right)^{\frac{1}{p}}=0.\label{DefAlmAuom8Pap3}
\end{eqnarray*}
\end{itemize}
\end{theorem}
\begin{remark}\cite{Moi}\label{RemCompSpAp}
$ $\\
\textbf{(i)} Every almost automorphic function is $S^p$-almost automorphic for $ 1\leq p < \infty $.\\
\textbf{(ii)} For all $ 1\leq p_1 \leq p_2 < \infty $, if $f$ is $S^{p_2}$-almost automorphic, then $f$ is $S^{p_1}$-almost automorphic. \\
\textbf{(iii)} The Bochner transform of an $X $-valued function is a $L^{p} (\left[ 0,1\right],X) $-valued function. Moreover, a function $f $ is $S^p$-almost automorphic if and only if $f^{b}$ is (Bochner) almost automorphic. \\
\textbf{(iv)} A function $ \varphi(t,s)$ for $t\in \mathbb{R},s\in [0,1]$ is the Bochner transform of a function $f$ (i.e., $ \exists$ $f:\mathbb{R}\longrightarrow X$ such that $(f^b (t))(s)=\varphi(t,s),\; t\in \mathbb{R}, s\in \left[ 0,1\right]$) if and only if $ \varphi(t+\tau,s-\tau)=\varphi(t,s)$ for all $t\in \mathbb{R}, s\in \left[ 0,1\right]$ and $\tau \in [s-1,s].$ 
\end{remark} 
\begin{definition}
Let $ 1\leq p< \infty $. A  function $ f: \mathbb{R}\times X\longrightarrow Y$ such that $f(\cdot,x)\in L^{p}_{loc} (\mathbb{R},Y) $  for each $ x\in X $ is said to be almost automorphic in $ t $ uniformly with respect to $ x $ in $X$ if for each compact set $K$ in $X$, for every sequence $(\sigma_{n})_{n \geq 0}$ of real numbers, there exists a subsequence $(s_{n})_{n \geq 0}\subset (\sigma_{n})_{n \geq 0} $ and $g(\cdot, x)\in L^{p}_{loc} (\mathbb{R},Y) $ for each $ x\in X $, such that 
\begin{eqnarray}
 \lim_{n} \sup_{x\in K}\left( \int_{t}^{t+1}\| f(s+s_{n},x)-g(s,x)\|^{p}ds\right)^{\frac{1}{p}} =0,  \lim_{n} \displaystyle \sup_{x\in K} \left( \int_{t}^{t+1}\| g(s-s_{n},x)-f(s,x)\|^{p}ds\right)^{\frac{1}{p}}=0\nonumber \\
 \label{DefAlmAuomUSubv}
\end{eqnarray}are well-defined for each $ t\in \mathbb{R}  .$\\ 
The space of all such functions is denoted by $ AAS^{p}U(\mathbb{R}\times X,Y) .$ 
\end{definition}

\subsection{ A subvariant minimizing functional}\label{section22Pep3}
The results in this section was introduced by \cite{CieuEzz} as a generalisation of the works by Fink, see \cite{Fink}.\\

Let $K$ be a compact subset of $X$. By $\mathcal{F}_{K}$, we denote the set of mild solutions $x$ on $\mathbb{R}$ of equation \eqref{Eq1SubStep} with relatively compact range and by $C_{K}(\mathbb{R},X)$ the set
$$
 C_{K}(\mathbb{R},X)=\{x\in C(\mathbb{R},X)\ :  \quad x  (t)\in K\;  \text{ for all }t\in
\mathbb{R} \}.
$$
Note that $ \mathcal{F}_{K} \subset C_{K}(\mathbb{R},X) $.
\begin{definition}\label{DefinitionSubvarFunct}
A functional $\lambda_{K}:C_{K}(\mathbb{R},X)\rightarrow\mathbb{R}$ is called a
subvariant functional associated to the compact set $K$, if
$\lambda_{K}$ satisfies the following conditions:\\
\textbf{(i)} $\lambda_{K}$ is invariant by translation: $\lambda_{K}(x_{\tau})=\lambda
_{K}(x)$ for each $\tau\in\mathbb{R}$, where $x_{\tau}(.)=x(\tau+.)$. \\
\textbf{(ii)} $\lambda_{K}$ is lower semicontinuous for the topology of
compact convergence: if $\displaystyle  \lim_{n\rightarrow+\infty}x_{n}=y$
uniformly on each compact subset of $\mathbb{R}$, then $\displaystyle
\lambda_{K}(y)\leq\liminf_{n\rightarrow+\infty}\lambda_{K}(x_{n})$. 
\end{definition}
\begin{definition}
 A function $x_{\ast}: \mathbb{R}\longrightarrow X$ is
called a minimal $K$-valued solution of equation \eqref{Eq1SubStep} if the following hold
\begin{equation}
x_{\ast}\in\mathcal{F}_{K}\quad\mathrm{and}\quad\lambda_{K}(x_{\ast})=\inf_{x\in\mathcal{F}_{K}}\lambda_{K}(x) \label{}
\end{equation}
\end{definition}
An example of subvariant functional is given by: 
\[
\lambda_{K}(x)=\sup_{t\in\mathbb{R}}\Phi(x(t))\text{\ \ \ where\ \ }\Phi\in
C(K,\mathbb{R}).
\]
In particular, we have
\[
\lambda_{K}(x)=\sup_{t\in\mathbb{R}}\parallel x(t)\parallel .
\]
%
\section{Almost automorphic solutions of equation \eqref{Eq1SubStep}}\label{section3}
In this section, we prove the existence of compact almost automorphic solutions to the semilinear evolution equation \eqref{Eq1SubStep}. \\

By a mild solution of equation \eqref{Eq1SubStep}, we mean a continuous function $ x: \mathbb{R} \longrightarrow X $ satisfying the following variation of constants formula:
\begin{eqnarray}
x(t)=T(t-\sigma)x(\sigma)+ \int_{\sigma}^{t} T(t-s)f(s,x(s))ds \qquad \mbox{for all} \; t\geq \sigma . \label{defMildSol}
\end{eqnarray}
Now, we give the following hypotheses:\\
\textbf{(H1)} $ (A,D(A)) $ is the infinitesimal generator of a $ C_{0} $-semigroup $ (T(t))_{t\geq 0} $ on $ X $.\\
\textbf{(H2)}  $ (T(t))_{t\geq 0} $ is compact, i.e.,  $  T(t)  $ is a compact operator for each $ t>0 $.\\
 \textbf{(H3)}  For all $R>0$, the function $\displaystyle \sup_{\parallel x\parallel\leq R} \| f(\cdot,x)\| \in L^{p}_{loc}(\r , \r) $ such that   $$\displaystyle               \sup_{t\in \r} \left( \int_{t}^{t+1}\sup_{\parallel x\parallel\leq R} \parallel f(s,x)\parallel^{p} ds\right)^{\frac{1}{p}} <+\infty .$$\\
  \textbf{(H4)} The function $ f $ is given by $f(t,x):= \varphi(t)g(x)+\psi(t)$, where $\varphi \in AAS^{1}(\r,\r) \cap L^{\infty}(\r, \r)$, $\psi \in AAS^{1}(\r,X)  $ and $ g \in C(X,X) $ transform bounded subsets of $X$ into bounded. \\
 
 \begin{remark}\label{Remark31}
 Hypothesis \textbf{(H4)} implies  \textbf{(H3)}.
 \end{remark} 
\begin{definition}\label{DefSpUnifContPaper3}
Let $ p \in [1, \infty)$ and $ f:\mathbb{R}\times X\longrightarrow X $ be such that $ f(\cdot, x) \in BS^{p}(\mathbb{R}, X) $ for each $ x\in X$. 
Then, $ f$ is said to be $ S^{p} $-uniformly continuous with respect to the second argument on each compact $K\subset X$ if, for all $K\subset X$ compact for all $ \varepsilon>0 $ there exists $ \delta_{K, \varepsilon} $ such that for all $ x_{1}, x_{2} \in K $, we have 
\begin{eqnarray}
\| x_{1}-x_{2}\| \leq \delta_{K, \varepsilon} \Longrightarrow \left( \int_{t}^{t+1}\|f(s,x_{1})-f(s,x_{2}) \|^{p}ds\right) ^{\frac{1}{p}} \leq \varepsilon \quad \text{for all} \; t\in \mathbb{R}  . \label{ComCharSpPaper3}
\end{eqnarray}
\end{definition}
\begin{lemma}\label{LemApSpCom1Paper3}
Let $ p \in [1, \infty)$ and $ f:\mathbb{R}\times X\longrightarrow X $ be such that $ f(\cdot, x) \in L^{p}_{loc}(\mathbb{R}, X) $ for each $ x\in X$. Then, $ f\in AAS^{p}U(\mathbb{R}\times X,X) $ if and only if the following hold:\\
\textbf{(i)} For each $ x\in X $, $ f(\cdot, x) \in AAS^{p}(\mathbb{R},X) .$\\
\textbf{(ii)} $ f$ is $ S^{p} $-uniformly continuous with respect to the second argument on each compact $K\subset X$. 
\end{lemma}

\begin{proof}
Let $ f\in AAS^{p}U(\mathbb{R}\times X,X) $ and $ f^{b}:\mathbb{R}\times X\longrightarrow L^{p}([0,1],X)  $ be the Bochner transform associated to $f$. It follows in view of  \cite[Proposition 5.5]{Ezz2},  that \textbf{(i)} is clearly satisfied and that: for each compact $K\subset X,$ for all $ \varepsilon>0 $ there exists $ \delta_{K, \varepsilon} $ such that  
for all $ x_{1}, x_{2} \in K $ one has 
\begin{eqnarray*}
\| x_{1}-x_{2}\| \leq \delta_{K, \varepsilon} &\Longrightarrow & \| f^{b}(t,x_{1})-f^{b}(t,x_{2}) \|_{p} \leq \varepsilon \quad \text{for all} \; t\in \mathbb{R} .
\end{eqnarray*}
Since
\begin{eqnarray*}
  \| f^{b}(t,x_{1})-f^{b}(t,x_{2}) \|_{p} &= &
 \left( \int_{[0,1]}\|(f^{b}(t,x_{1}))(s)-(f^{b}(t,x_{2}))(s) \|^{p}ds\right) ^{\frac{1}{p}}  \\ &=& 
 \left( \int_{t}^{t+1}\|f(s,x_{1})-f(s,x_{2}) \|^{p}ds\right) ^{\frac{1}{p}} \quad \text{for all} \; t\in \mathbb{R}.\\ 
\end{eqnarray*}
It follows that \eqref{ComCharSpPaper3} holds and then \textbf{(ii)} is achieved.\\ 
Conversely, let $ f:\mathbb{R}\times X\longrightarrow X $ be a function such that $ f(\cdot, x) \in L^{p}_{loc}(\mathbb{R}, X) $ for each $ x\in X . $ Assume that $ f $ satisfy \textbf{(i)-(ii)}. Let us fix a compact subset $ K $ in $X$ and $ \varepsilon>0 .$ Since $ K $ is compact, it follows that there exists a finite subset $ \lbrace x_{1},...,x_{n} \rbrace  \subset K$ ($ n\in \mathbb{N}^{*} $) such that $ \displaystyle{K\subseteq \bigcup_{i=1}^{n} B(x_{i}, \delta_{K, \varepsilon} )} .$ For $ x\in K $ there exist $ i=1,...,n $ satisfying $ \|x-x_{i}\|\leq \delta_{K, \varepsilon} $.
Let $(\sigma_{n})_{n \geq 0}$ be a sequence of real numbers and its subsequence $(s_{n})_{n \geq 0} $ such that
\begin{eqnarray}
 & &\left( \int_{t}^{t+1}\| f(s+s_{l}-s_{k} ,x)-f(s,x) \|^{p}ds\right) ^{\frac{1}{p}} \leq   \left( \int_{t}^{t+1}\| f(s+s_{l}-s_{k} ,x)-f(s+s_{l}-s_{k} , x_{i}) \|^{p}ds\right) ^{\frac{1}{p}} \nonumber \\ & & + \left( \int_{t}^{t+1}\| f(s+s_{l}-s_{k} ,x_{i})-f(s,x_{i}) \|^{p}ds\right) ^{\frac{1}{p}}+\left( \int_{t}^{t+1}\| f(s ,x_{i})-f(s,x) \|^{p}ds\right) ^{\frac{1}{p}}. 
 \label{ComResFor1}
\end{eqnarray}
Using \textbf{(i)}, it holds that $ f(\cdot,x_{i}) \in AAS^{p}(\mathbb{R},X) $. Hence, for $ k,l $ large enough, for each $t\in \r$
\begin{eqnarray}
 \left( \int_{t}^{t+1}\| f(s+s_{l}-s_{k} ,x_{i})-f(s,x_{i}) \|^{p}ds\right) ^{\frac{1}{p}} \leq \frac{\varepsilon}{3}. \label{ComResFor2}
\end{eqnarray} 
Otherwise, since $\| x-x_{i}\| \leq \delta_{K,\delta} $ and  by \textbf{(ii)} we claim that, for each $k,l \in \mathbb{N} $
\begin{eqnarray}
 \left( \int_{t}^{t+1}\| f(s+s_{l}-s_{k} ,x)-f(s+s_{l}-s_{k} , x_{i}) \|^{p}ds\right) ^{\frac{1}{p}} \leq \frac{\varepsilon}{3} \label{ComResFor3} \quad \text{for all} \; t\in \mathbb{R}
\end{eqnarray}
and 
\begin{eqnarray}
 \left( \int_{t}^{t+1}\| f(s ,x)-f(s, x_{i}) \|^{p}ds\right) ^{\frac{1}{p}} \leq \frac{\varepsilon}{3} \quad \text{for all} \; t\in \mathbb{R}. \label{ComResFor4}
\end{eqnarray}
Consequently, if we replace \eqref{ComResFor2}, \eqref{ComResFor3} and \eqref{ComResFor4} in \eqref{ComResFor1} then, for $ k,l $ large enough we have  
\begin{eqnarray*}
\sup_{x\in K}\left( \int_{t}^{t+1}\| f(s+s_{l}-s_{k} ,x)-f(s,x) \|^{p}ds\right) ^{\frac{1}{p}} \leq \varepsilon. 
\end{eqnarray*}
for each $ t\in \mathbb{R} $. $\ \blacksquare$
\end{proof}
\begin{remark}\label{RemarkContSpG3}
\textbf{(a)} In view of Lemma \ref{LemApSpCom1Paper3}. If $f$ is the function defined by \textbf{(H4)}, then  $ f\in  AAS^{1}U(\r \times X,X) $.\\
\textbf{(b)} Using Lemma \ref{LemApSpCom1Paper3}, we can prove easily that the function $ g $ in  \eqref{DefAlmAuomUSubv} is $ S^{p} $-uniformly continuous with respect to the second argument on each compact  $K\subset X$.
\end{remark}
\begin{lemma}\label{Lemma32BoundFandGStep3}
 Let $ f\in AAS^{p}U(\mathbb{R}\times X,X) $ for $ p \in [1, \infty)$. Then, for each compact $ K $ in $X$, we have\\
\textbf{(i)} $k_{p}=\sup_{t\in \r}\sup_{x\in K}\left( \int_{t}^{t+1}\parallel f(s,x)\parallel^{p} ds \right)^{\frac{1}{p}}  <+\infty$.\\
\textbf{(ii)} $l_{p}=\sup_{t\in \r}\sup_{x\in K}\left( \int_{t}^{t+1}\parallel g(s,x)\parallel^{p} ds \right)^{\frac{1}{p}}  <+\infty$, where $ g $ is the function defined in \eqref{DefAlmAuomUSubv}.\\
\end{lemma}
\begin{proof}
\textbf{(i)} Let $ K  \subset X$ be compact and let $ \delta>0$ . Then, there exists $ \lbrace x_{1},...,x_{n} \rbrace  \subset K$ ($ n\in \mathbb{N}^{*} $) such that $ \displaystyle{K\subseteq \bigcup_{i=1}^{n} B(x_{i}, \delta )} .$ Therefore, for each $ x\in K $ there exist $ i=1,...,n $ satisfying $ \|x-x_{i}\|\leq \delta $. Then, for all $ x\in K $ and $ \varepsilon >0 $, we have
\begin{eqnarray*}
\left( \int_{t}^{t+1}\parallel f(s,x)\parallel^{p} ds \right)^{\frac{1}{p}} \leq \left( \int_{t}^{t+1}\parallel f(s,x)-f(s,x_{i})\parallel^{p} ds \right)^{\frac{1}{p}}+\sum_{i=1}^{n}\left( \int_{t}^{t+1}\parallel f(s,x_{i})\parallel^{p} ds \right)^{\frac{1}{p}}
\end{eqnarray*}
$\text{ for all } t\in \mathbb{R}.$
Hence, by Lemma \ref{LemApSpCom1Paper3}, we claim that
\begin{eqnarray*}
\sup_{t\in \r}\sup_{x\in K} \left( \int_{t}^{t+1}\parallel f(s,x)\parallel^{p} ds \right)^{\frac{1}{p}} \leq  \varepsilon+ \sum_{i=1}^{n} \sup_{t\in \r} \left( \int_{t}^{t+1}\parallel f(s,x_{i})\parallel^{p} ds \right)^{\frac{1}{p}} < \infty.
\end{eqnarray*}
\textbf{(ii)} Let $ K  \subset X$ be compact and let  $(\sigma_{n})_{n \geq 0}$ be a sequence. Since $ f\in AAS^{p}U(\mathbb{R}\times X,X) $, it follows that there exists a subsequence $(s_{n})_{n \geq 0}\subset (\sigma_{n})_{n \geq 0} $ such that \eqref{DefAlmAuomUSubv} holds. Hence, for all $ n  \in \n  $
\begin{eqnarray*}
\left( \int_{t}^{t+1}\parallel g(s,x)\parallel^{p} ds \right)^{\frac{1}{p}} &\leq &\sup_{x\in K}\left( \int_{t}^{t+1}\| f(s+s_{n},x)-g(s,x)\|^{p}ds\right)^{\frac{1}{p}}\\ &&+\sup_{x\in K}\left( \int_{t}^{t+1}\| f(s+s_{n},x)\|^{p}ds\right)^{\frac{1}{p}} \\ &\leq &\sup_{x\in K}\left( \int_{t}^{t+1} \| f(s+s_{n},x)-g(s,x)\|^{p}ds\right)^{\frac{1}{p}} +k_{p} \text{ for all }  x\in K, t\in \mathbb{R}.
\end{eqnarray*}
Therefore, for $ n $ large enough, we have 
\begin{eqnarray*}
\sup_{x\in K} \left( \int_{t}^{t+1}\parallel g(s,x)\parallel^{p} ds \right)^{\frac{1}{p}}\leq k_{p} \quad \text{     for all }   t\in \mathbb{R}.
\end{eqnarray*}
This proves the result.$\ \blacksquare$
\end{proof}
\begin{lemma}\cite[Theorem 4.1.2]{aubin2011applied}\label{lem: Banach Steinhaus}
Let $Y$ be a normed space and $\left(T_{i}\right)_{i \in T}$ be any family
of bounded linear operators on $Y$ such that $\sup_{i \in I}\norm{T_{i}}<\infty$.
If $D$ is a dense subset of $Y$, and if for each $y\in D$
\[
T_{i}y\to Ty\,\,\,\,\mbox{as}\,\,i\to\infty,
\]
 for some bounded linear operator $T$. Then for every compact set
$K\subset Y$
\[
\sup_{y\in K}\norme{T_{i}y-Ty}\to0\,\,\,\,\mbox{as}\,\,i\to\infty.
\]
\end{lemma}

The following Lemma is needed in the proof of Theorem \ref{Lemma31}.
\begin{lemma}
\label{lem:subsubsection} Let $(x_{n})_{n}$ be a sequence in a Banach
space $X$ and $x\in X$. If every subsequence $(x'_{n})_{n}\subset(x_{n})_{n}$
has a subsequence $(x''_{n})_{n}\subset(x'_{n})_{n}\subset(x_{n})_{n}$
that converges to $x$, then the whole sequence $(x_{n})_{n}$
converges to $x$.
\end{lemma}
\begin{proof}
By contradiction. $\ \blacksquare$
\end{proof}
\begin{theorem}\label{Lemma31}
Let $ f\in AAS^{p}U(\mathbb{R}\times X,X) $ for $ p \in [1, \infty)$. Assume that \textbf{(H1)} holds. Then, a mild solution of equation \eqref{Eq1SubStep} with relatively compact range is uniformly continuous.

\end{theorem}
\begin{proof} Let $x$ be a solution of equation \eqref{Eq1SubStep} such that  $K=\overline{\left\{ x(t):\,t\in\r\right\} } \subset X$ is compact. We distinguish two cases:\\
\textbf{The case $ p\in (1,\infty) .$}  Since $  (T(t))_{t\geq 0}$ is a $C_{0}$-semigroup on $X$, it follows in view of Lemma \ref{lem: Banach Steinhaus}, that
\begin{eqnarray}
\lim_{t\rightarrow 0}\sup_{x\in K}\parallel T(t)x-x\parallel=0. \label{BSteiHau}
\end{eqnarray}
Furthermore, the $C_{0}$-semigroup $  (T(t))_{t\geq 0}$ on $X$ satisfies $$ \|  T(t) \|  \leq Me^{\omega t}\qquad \text{ for }t \geq 0 $$
for some $ M\geq 1 $, $ \omega \in \r $. Let $ q $ be such that $ \dfrac{1}{p}+ \dfrac{1}{q}=1$.  Using H\"{o}lder inequality, we have for all $ t, s \in \r $,

\begin{eqnarray*}
\parallel x(t)-x(s)\parallel &\leq & \parallel T(t-s)x(s)-x(s)\parallel+\int_{s}^{t}\parallel T(t-\sigma) f(\sigma,x(\sigma))\parallel\ d\sigma \\ &\leq & \sup_{x\in K}\parallel T(t-s)x-x\parallel +M\left( \int_{s}^{t}e^{q\omega(t-\sigma)}\ d\sigma\right)^{\frac{1}{q}}\left( \int_{s}^{t}\parallel f(\sigma,x(\sigma))\parallel ^{p} d\sigma\right)^{\frac{1}{p}}\\ &\leq & \sup_{x\in K}\parallel T(t-s)x-x\parallel +k_{p} M\left( t-s+3\right)^{\frac{1}{p}} \left( \int_{0}^{t-s}e^{q\omega \sigma}\ d\sigma\right)^{\frac{1}{q}}, \quad t\geq s.
\end{eqnarray*}
Then, by interchanging $s$ and $t$, we claim that
\begin{eqnarray*}
\parallel x(t)-x(s)\parallel  &\leq & \sup_{x\in K}\parallel T(\mid t-s \mid)x-x\parallel + k_{p} M^{'}  \left( \mid t-s\mid+3\right)^{\frac{1}{p}} \left( \int_{0}^{\mid t-s\mid}e^{q\omega \sigma}\ d\sigma\right)^{\frac{1}{q}}. 
\end{eqnarray*}
From $\displaystyle  \left( \int_{0}^{\mid t-s\mid}e^{q\omega \sigma}\ d\sigma\right)^{\frac{1}{q}} \rightarrow 0$ as $ \mid t-s\mid \rightarrow 0 $ and  \eqref{BSteiHau} we deduce that $$ \parallel x(t)-x(s)\parallel \rightarrow 0 \; \text{ as} \; \mid t-s\mid \rightarrow 0 .$$
Consequently, $x$ is uniformly continuous. \\
\textbf{The case $ p=1.$} To show that $x$ is uniformly continuous, we take two real sequences $(t_{n})_{n}$
and $(s_{n})_{n}$ such that $\norm{t_{n}-s_{n}}\to0$ as $n\to\infty$
and we prove that $y_{n}=x(t_{n})-x(s_{n})\to0$ as $n\to\infty$. In
fact, consider the sequence $\left(h_{n}\right)_{n}$ defined by $h_{n}=t_{n}-s_{n}$.
Then, we have $\lim_{n\to\infty}h_{n}=0$. Assume without loss of generality
that $0 \leq h_{n}\leq 1$ for all $n\in\n$. Thus, we have
\begin{eqnarray*}
x(t_{n})=x(s_{n}+h_{n}) & = & T(h_{n})x(s_{n})+\int_{s_{n}}^{s_{n}+h_{n}}T(s_{n}+h_{n}-s)f\left(s,x(s)\right)ds\\
 & = & T(h_{n})x(s_{n})+\int_{0}^{h_{n}}T(h_{n}-s)f\left(s+s_{n},x(s+s_{n})\right)ds.
\end{eqnarray*}

\noindent Let $M\geq1$ and $\omega\in\r$ such that $\norme{T\left(t\right)}\leq Me^{\omega t}$
for all $t\geq0$. Then
\begin{eqnarray}
\norme{x(s_{n}+h_{n})-x(s_{n})} & \leq & \norme{T(h_{n})x(s_{n})-x(s_{n})}+\int_{0}^{h_{n}}\norme{T(h_{n}-s) f\left(s+s_{n},x(s+s_{n})\right)}ds\nonumber \\
 & \leq & \sup_{y\in K}\norme{T(h_{n})y-y}+M\int_{0}^{h_{n}}e^{\omega(h_{n}-s)}\norme{f\left(s+s_{n},x(s+s_{n})\right)}ds\nonumber \\
 & \leq & \sup_{y\in K}\norme{T(h_{n})y-y}+Me^{\norm{\omega}h_{n}}\int_{0}^{h_{n}}\norme{f\left(s+s_{n},x(s+s_{n})\right)}ds.\label{eq:uniform continuity main inequality}
\end{eqnarray}
 The semigroup $\left(T(t)\right)_{t\geq0}$ is strongly continuous, then for each $y\in K$, $T(h_{n})y\to y$ as, $n\to\infty$. This
implies that $\sup_{n}\|T(h_{n})y\|<\infty$ for each $y\in X$
and thus by the Banach-Steinhaus's Theorem $\sup_{n}\|T(h_{n})\|<\infty$.
It follows from Lemma \ref{lem: Banach Steinhaus} that 
\begin{equation}
\sup_{y\in K}\norme{T(h_{n})y-y}\to0\,\,\,\,\,\mbox{as }n\to\infty.\label{eq:Banach-Steinhaus}
\end{equation}
Let $y'_{n}=x(s'_{n}+h'_{n})-x(s'_{n})$ be a subsequence of $y_{n}$.
From the $S^{1}$-almost automorphy of $f$, there exist a subsequence
$\left(s''_{n}\right)_{n}\subset\left(s'_{n}\right)_{n}$ and there
exist $g:\r\times X\to X$ with $g(\cdot,x) \in L^{1}_{loc}(\r ,X)$ for each $ x\in X $ such that for every $t\in\r$, we have 
\begin{equation}
\sup_{x\in K} \int_{t}^{t+1}\norme{f\left(s+s''_{n},x\right)-g\left(s,x\right)}ds\to0\,\,\,\,\mbox{as }n\to\infty.\label{eq:SAA for uniform continuity}
\end{equation}
Let $\left(h''_{n}\right)_{n}$ be the corresponding subsequence of
$\left(h'_{n}\right)_{n}$. We can assume that $0\leq h_{n}\leq1$
for all $n\in\n$. Then, we have 
\begin{eqnarray*}
\int_{0}^{h''_{n}}\norme{f\left(s+s''_{n},x(s+s''_{n})\right)}ds & \leq & \int_{0}^{h''_{n}}\norme{f\left(s+s''_{n},x(s+s''_{n})\right)-g\left(s,x(s+s''_{n})\right)}ds \\ &+&\int_{0}^{h''_{n}}\norme{g\left(s,x(s+s''_{n})\right)}ds\\
 & \leq & \sup_{x\in K}\int_{0}^{h''_{n}}\norme{f\left(s+s''_{n},x\right)-g\left(s,x\right)}ds+\int_{0}^{h''_{n}}\norme{g\left(s,x(s+s''_{n})\right)}ds\\
 & \leq & \sup_{x\in K}\int_{0}^{1}\norme{f\left(s+s''_{n},x\right)-g\left(s,x\right)}ds+\int_{0}^{h''_{n}}\norme{g\left(s,x(s+s''_{n})\right)}ds.
\end{eqnarray*}
Using (\ref{eq:SAA for uniform continuity}) we claim that ${\displaystyle \lim_{n\rightarrow \infty} \sup_{x\in K} \int_{0}^{1}\norme{f\left(s+s''_{n},x\right)-g\left(s,x\right)}ds}=0$. \\
On the other hand,  let $ \delta>0$ . Then, there exist  $  x_{1},...,x_{m}  \in K$, $ m\in \mathbb{N}^{*} $, such that , for each  $ x\in K $  there exists $ i_{0}=1,...,m $ satisfies $\| x-x_{i_{0}} \| \leq \delta $. Let $ \varepsilon>0 $, then, for all $ s \in \r $ for all $ n \in \n  $, there exists $ i(s,n)=1,...,m $ with $\| x(s+s''_{n})- x_{i(s,n)} \| \leq \delta $. Hence, by Remark \ref{RemarkContSpG3}, we obtain
\begin{eqnarray*}
\int_{0}^{h''_{n}}\norme{g\left(s,x(s+s''_{n})\right)}ds &\leq & \int_{0}^{1}\norme{g\left(s,x(s+s''_{n})\right) -g(s,x_{i(s,n)})}ds+\sum_{i=1}^{m}\int_{0}^{h''_{n}}\norme{g\left(s,x_{i}\right)}ds \\ &\leq & \varepsilon + \sum_{i=1}^{m}\int_{0}^{h''_{n}}\norme{g\left(s,x_{i}\right)}ds  \quad \text{ for all } \; \varepsilon>0.
\end{eqnarray*}
By Lemma \ref{Lemma32BoundFandGStep3}, we claim that, for each $ i=1,...,m $,  we have $${\displaystyle \int_{0}^{h''_{n}}\norme{g\left(s,x_{i}\right)}ds\to0} \; \; \text{ as } n\to\infty.$$
Then, ${\displaystyle \int_{0}^{h''_{n}}\norme{g\left(s,x(s+s''_{n})\right)}ds \to0}$
as $n\to\infty$. Therefore, 
\begin{equation}
\int_{0}^{h''_{n}}\norme{f\left(s+s''_{n},x(s+s''_{n})\right)}ds\to0\,\,\,\,\,\mbox{as }n\to\infty.\label{eq: second convergence}
\end{equation}
Consequently, we conclude from (\ref{eq:uniform continuity main inequality}),
(\ref{eq:Banach-Steinhaus}) and (\ref{eq: second convergence}) that
\[
\norme{y''_{n}}=\norme{x(t''_{n})-x(s''_{n})}=\norme{x(s''_{n}+h''_{n})-x(s''_{n})}\to0\,\,\,\,\,\mbox{as }n\to\infty,
\]
Finally, using Lemma \ref{lem:subsubsection}, the whole sequence $y_{n}=x(t_{n})-x(s_{n})$
converges to $0$. We conclude that the mild solution $x$ is uniformly continuous.$\ \blacksquare$\\
\end{proof}

\begin{lemma}\label{Propostion1}
Let $ f\in AAS^{p}(\r \times X,X)  $ for $ p\in [1, \infty) $. Assume that \textbf{(H1)-(H2)} hold. If equation \eqref{Eq1SubStep} has at least a mild solution
$x_{0}: [t_{0},+\infty) \longrightarrow X $ with relatively compact range. Then,
there exists a mild solution $x$ on $\mathbb{R}$
of equation \eqref{Eq1SubStep} such that 
\begin{eqnarray}
\lbrace x(t):\; t \in \mathbb{R} \rbrace \subset \overline{ \lbrace x_{0}(t):\; t\geq t_{0} \rbrace }. \label{MildSolCompRange}
\end{eqnarray}
\end{lemma}
\begin{proof} 
 Let $\displaystyle
K:=\overline{\{x_{0}(t) :\ t\geq t_{0}\}} \subset X$. By assumption $ K $ is compact. Moreover, the mild solution $x_{0}$ is given by:
$$
x_{0}(t)=T(t-s)x_{0}(s)+\int_{s}^{t}T(t-\sigma)f(\sigma,x_{0}(\sigma
))\ d\sigma\qquad\text{ for }t\geq s\geq t_{0}.
$$
Let $(t_{n}^{\prime})_{n}$ be a sequence of real numbers such that $$\lim_{n\rightarrow+\infty}t_{n}^{\prime}=+\infty.$$
Since $f \in AAS^p U(\mathbb{R}\times X,X)$, it follows that 
there exist a subsequence $(t_{n})_{n}\subset (t_{n}^{\prime})_{n}$ and $g:\mathbb{R}\times K\rightarrow X$ such that for each $ t\in \r $, we have
\begin{eqnarray}
 \lim_{n} \sup_{x\in K}\left( \int_{t}^{t+1}\| f(s+s_{n},x)-g(s,x)\|^{p}ds\right)^{\frac{1}{p}} =0. \label{DefAlmAuomUSubvHalf1}
\end{eqnarray}
For $n\in\mathbb{N}$ sufficiently large, the function $t  \mapsto x_{0}(\cdot +t_{n}) \in K$ is
defined on $\r $.  Let $u_{n}(t):=x_{0}(t+t_{n})$ for all $ t\in \mathbb{R} $ and the set $ U_{n}=\lbrace u_{n}: \;  n\in \mathbb{N}\rbrace $. Then, $U_{n}\subset
C(\mathbb{R},X)$ and $ (u_{n})_{n} $ satisfies for each $t\in\mathbb{R}$ $u_{n}(t)\in K$.
Therefore, for each $t\in\mathbb{R}$, the subset $ U_{n}(t)=\lbrace u_{n}(t): \;  n\in \mathbb{N}\rbrace $ is relatively compact in $X$. By Theorem \ref{Lemma31} and using Arzela-Ascoli's Theorem and an extraction diagonal argument applied to  the sequence $ (u_{n} )_{n} $, we  claim that there exist
$x_{\ast}\in C(\mathbb{R},X)$ and a subsequence of $(t_{n})_{n}$ such that
\begin{eqnarray} \label{ConvFormMildSol}
x_{0}(t+t_{n})\rightarrow x_{\ast}(t)\qquad\mathrm{as}\qquad n\rightarrow
+\infty,
\end{eqnarray}
uniformly on each compact subset of $\mathbb{R}$. Hence, for $n\in\mathbb{N}$ sufficiently large, we have
\begin{eqnarray} \label{MildSolEqUn}
x_{0}(t+t_{n})=T(t-s)x_{0}(s+t_{n})+\int_{s}^{t}T(t-\sigma)f(\sigma
+t_{n},x_{0}(\sigma+t_{n}))\ d\sigma, \quad t\geq s.
\end{eqnarray}
Let $ t,s\in \mathbb{R} $,  $t\geq s$. Then, 
using \eqref{ConvFormMildSol}, we claim that 
\begin{eqnarray}
\displaystyle \lim_{n\rightarrow+\infty}T(t-s) x_{0}(s+t_{n})=T(t-s)x_{\ast}(s). \label{LimitVStepPart1}
\end{eqnarray}
Now, we prove that 
\begin{eqnarray}
\lim_{n\rightarrow+\infty}\int_{s}^{t}T(t-\sigma)f(\sigma+t_{n},x_{0}(\sigma
+t_{n}))\ d\sigma=\int_{s}^{t}T(t-\sigma)g(\sigma,x_{\ast}(\sigma))\ d\sigma . \label{LimitVStep}
\end{eqnarray}
In fact, we have
\begin{eqnarray*}
& &\left( \int_{t}^{t+1}\parallel f(\sigma +t_{n},x_{0}(\sigma +t_{n}))-g(\sigma,x_{\ast}(\sigma))\parallel^{p} d\sigma\right)^{\frac{1}{p}} \\ &\leq &\left( \int_{t}^{t+1} \parallel f(\sigma +t_{n},x_{0}(\sigma +t_{n}))-g(\sigma,x_{0}(\sigma +t_{n}))\parallel^{p} d\sigma \right)^{\frac{1}{p}} \\ &+&\left( \int_{t}^{t+1}\ \parallel g(\sigma ,x_{0}(\sigma +t_{n}))-g(\sigma,x_{\ast}(\sigma))\parallel^{p} d\sigma\right)^{\frac{1}{p}}  \\  &\leq & \underbrace{\sup_{x\in K}\left( \int_{t}^{t+1}\parallel f(\sigma +t_{n},x)-g(\sigma ,x)\parallel^{p}d\sigma\right)^{\frac{1}{p}}}_{\textbf{(I)}}  \\ &+&\underbrace{\left( \int_{t}^{t+1}\parallel
g(\sigma ,x_{0}(\sigma +t_{n}))-g(\sigma ,x_{\ast}(\sigma))\parallel^{p}d\sigma \right)^{\frac{1}{p}}}_{\textbf{(II)}} , \; t\in \mathbb{R}.
\end{eqnarray*}
Using \eqref{DefAlmAuomUSubvHalf1}, we obtain that \textbf{(I)}$\rightarrow 0  $ as $ n\rightarrow \infty $. On the other hand, by Remark \ref{RemarkContSpG3}, it follows that $g $ is $ S^{p} $-uniformly continuous. Therefore, we have \textbf{(II)}$\rightarrow 0  $ as $ n\rightarrow \infty $. Consequently, for each $t\in \r$
\begin{eqnarray}
K_{p}(t, n):=\left( \int_{t}^{t+1}\parallel f(\sigma +t_{n},x_{0}(\sigma +t_{n}))-g(\sigma,x_{\ast}(\sigma))\parallel^{p} d\sigma\right)^{\frac{1}{p}} \rightarrow 0 \text{ as } n\rightarrow \infty. \label{LimitLemma32SteV}
\end{eqnarray}
Thus, for $p \in [1, \infty)$, we have
\begin{eqnarray*}
& &\| \int_{s}^{t} T(t-\sigma)f(\sigma+t_{n},x_{0}(\sigma+t_{n}
))d\sigma -\int_{s}^{t} T(t-\sigma)g(\sigma,x_{\ast}(\sigma)) d\sigma \|  \\ 
&\leq &  \int_{s}^{t} \| T(t-\sigma)\left[ f(\sigma+t_{n},x_{0}(\sigma+t_{n}))-g(\sigma,x_{\ast}(\sigma))\right]  \| d\sigma  \\ 
&\leq & M \sum_{m=[s]}^{[t]+1} e^{|\omega |(t-m)} \int_{m}^{m+1}\parallel f(\sigma+t_{n},x_{0}(\sigma+t_{n}))-g(\sigma,x_{\ast}(\sigma)) \parallel d\sigma 
\\ &\leq & M \sum_{m=[s]}^{[t]+1} e^{|\omega |(t-m)} \left( \int_{m}^{m+1}\parallel f(\sigma+t_{n},x_{0}(\sigma+t_{n}))-g(\sigma,x_{\ast}(\sigma)) \parallel^{p} d\sigma\right)^{\frac{1}{p}} 
\\ &=& M  \sum_{m=[s]}^{[t]+1} e^{|\omega |(t-m)} k_{p}(m,n)
\end{eqnarray*}
Therefore, using the limit in \eqref{LimitLemma32SteV} we obtain \eqref{LimitVStep}.
Consequentely, by \eqref{LimitVStepPart1} and \eqref{LimitVStep} we deduce that
$$
x_{\ast}(t)=T(t-s)x_{\ast}(s)+\int_{s}^{t}T(t-\sigma)g(\sigma,x_{\ast}(\sigma))\ d\sigma
\qquad\text{ for }t\geq s. 
$$ 
Hence, $ x_{\ast} $ is a mild solution with a relatively compact range of the following equation 
\begin{equation}
x_{\ast}^{\prime}(t)=Ax_{\ast}(t)+g(t,x_{\ast}(t)), \quad t\in \r . \label{EquStarX}
\end{equation}
By taking the sequence $(-t_{n})_{n}$ instead of $(t_{n})_{n}$, we apply the same construction to the mild solution $ x_{\ast} $ of equation \eqref{EquStarX}. Then, we obtain the existence of a mild solution $x$ with relatively compact range of equation \eqref{Eq1SubStep}. Consequently, $x$ is a mild solution satisfying \eqref{MildSolCompRange}.  $\ \blacksquare$
\end{proof}
\subsection{Solutions with relatively compact ranges}\label{section31Pep3}
In this section, we investigate the relative compactness of the range of bounded solutions to equation \eqref{Eq1SubStep} on the right half-line. We distinguish, two cases, $p\in (1, \infty) $ and $ p=1 $.\\
To conclude our results, we use the Kuratowski measure of noncompactness  $\alpha $ of bounded subsets  $ B $ in $ X $ defined by: 
\begin{eqnarray}
\alpha(B):= \inf \lbrace \varepsilon>0: \, B\, \text{has a finite cover of balls with diameter}\, < \varepsilon  \rbrace . \label{MeaNComp}
\end{eqnarray}
The measure $ \alpha $ satisfies the following properties:\\
\textbf{(a)} $ \alpha(B)=0 \Leftrightarrow  B$ is relatively compact. \\ 
\textbf{(b)} $\alpha(B_{1}+B_{2})\leq  \alpha(B_{1})+\alpha(B_{2}) $ for all $ B_{1} $ and $ B_{2} $ bounded in $ X $.\\
\textbf{(c)} $\alpha(B(0,\varepsilon)= 2\varepsilon $ for all $ \varepsilon >0 .$\\
For more details, we refer to \cite{Lak}. \\
 \vspace*{0.3cm}\\ 
\textbf{The case $p\in (1, \infty)$.} 
\begin{lemma}\label{Lemma Relativ Comp p not 1}
Assume that \textbf{(H1)-(H3)} hold. If equation \eqref{Eq1SubStep} has at least a bounded mild solution
$x_{0}:[t_{0},+\infty) \longrightarrow X$. Then, its range $\{x_{0}(t)\ : \ t\geq t_{0}\}$ is relatively compact in $X$.\\
\end{lemma}
\begin{proof} 
\textbf{(i)} For $ 0<\varepsilon <1$,  define $$ x_{0}(t)=T(\varepsilon)x_{0}(t-\varepsilon) +\int_{t-\varepsilon}^{t}T(t-\sigma)f(\sigma,x_{0}(\sigma))\ d\sigma
\qquad\text{ for }t> t_{0}+1.   $$
Let $ t>t_{0}+1 $. Then, we have
\begin{eqnarray}
T(\varepsilon)x_{0}(t-\varepsilon) \in T(\varepsilon)  \overline{B}(0,M_{0}) \label{CompacityX01}
\end{eqnarray}
where $ M_{0}:=\sup_{t\geq t_{0}}\parallel x_{0}(t)\parallel<+\infty  $ and $ \overline{B}(0,M_{0}) $ is the closed ball of center $ 0 $ and radius $ M_{0}$.\\
Using H\"{o}lder inequality, we obtain 
 \begin{eqnarray*}
\| \int_{t-\varepsilon}^{t}T(t-\sigma)f(\sigma,x_{0}(\sigma))\ d\sigma \| \leq Me^{\mid \omega \mid \varepsilon} \ \varepsilon ^{\frac{1}{q}}\left( \int_{t-\varepsilon}^{t}\| f(\sigma,x_{0}(\sigma))\|^{p} d\sigma \right)^{\frac{1}{p}} \leq Me^{\mid \omega \mid \varepsilon} \ \varepsilon ^{\frac{1}{q}} \ \tilde{k}_{p}:= \delta(\varepsilon),
\end{eqnarray*}
where $ \tilde{k}_{p}= \displaystyle \sup_{t\in \mathbb{R}}   \left( \int_{t-1}^{t}\sup_{x\in  \overline{B}(0,M_{0}) } \| f(\sigma,x)\|^{p} d\sigma \right)^{\frac{1}{p}} $.\\
Hence, $$   \int_{t-\varepsilon}^{t}T(t-\sigma)f(\sigma,x_{0}(\sigma))\ d\sigma \in  \overline{B}(0,\delta(\varepsilon)) . $$
Consequently, 
$$ \lbrace x_{0}(t):\; t\geq t_{0} \rbrace \subset  \lbrace x_{0}(t):\; t\in  [t_{0}, t_{0}+1] \rbrace \cup T(\varepsilon)  \overline{B}(0,M_{0}) \cup \overline{B}(0,\delta(\varepsilon)) . $$
Thus, by the continuity of $  x_{0} $, the set $\lbrace x_{0}(t):\; t\in  [t_{0}, t_{0}+1] \rbrace $ is compact. Hence $$ \alpha (\lbrace x_{0}(t):\; t\in  [t_{0}, t_{0}+1] \rbrace )=0. $$Therefore, from hypothesis \textbf{(H2)}, we claim that $$ \alpha (\lbrace x_{0}(t):\; t\geq t_{0}\rbrace )\leq 2\delta(\varepsilon)\; \text{ for all } \; \varepsilon>0 .$$
Since  $ \delta(\varepsilon)\rightarrow 0 $ as $ \varepsilon \rightarrow 0 $.  It follows that $$ \alpha \left( \lbrace x_{0}(t):\; t\geq t_{0}\rbrace\right)  =0 .$$ 
\begin{flushright}
$\ \blacksquare$
\end{flushright}
\end{proof}
\textbf{The case $p=1$.}\\
In this case, the following Lemmas are needed.
\begin{lemma}\label{Lemma1CaseP1}
Let $ \rho \in AAS^{1}(\r,X) $. Then, for all $0<\varepsilon <1$, the following function:
\begin{eqnarray}
x_{ \rho}^{\varepsilon}(t):= \int_{t-\varepsilon}^{t} T(t-s) \rho(s) ds \quad \text{ for } \; t\geq t_{0}, \; t_{0}\in \r
\end{eqnarray}
has a relatively compact range on $ [t_{0}, \infty) .$
\end{lemma}
\begin{proof}
Let $0<\varepsilon <1$ and $ (t'_{n})_{n} \subset \r$ be a sequence. Since $  \rho \in AAS^{1}(\r,X) $ it follows that, there exist a subsequence $ (t_{n})_{n} \subset (t'_{n})_{n} $ and $ g \in L^{1}_{loc}(\r,X) $ such that for each $ t \in \r $
\begin{eqnarray}
  \int_{t}^{t+1}\|  \rho(t_{n}+s)-g(s)\|\ ds  \rightarrow 0 \; \text{ as } n\rightarrow \infty.\label{ProfLemP1Pap3}
\end{eqnarray}
Let $ y_{ \rho}^{\varepsilon}:=\displaystyle \int_{0}^{\varepsilon} T(\sigma)g(-\sigma) d \sigma \in X$. Then, 
\begin{eqnarray*}
\| x_{ \rho}^{\varepsilon}(t_{n})- y_{ \rho}^{\varepsilon}\| &=&\| \int_{t_{n}}^{t_{n}-\varepsilon} T(t_{n}-s)   \rho(s ) ds- \int_{0}^{\varepsilon} T(\sigma)g(-\sigma) d \sigma\| \\ &=& \| \int_{0}^{\varepsilon} T(\sigma) \left[   \rho(t_{n}-\sigma ) -g(- \sigma) \right]  d \sigma \| \\ &\leq & M e^{|\omega |} \int_{0}^{1} \|  \rho(t_{n}-\sigma ) -g(-\sigma)  \| d \sigma .
\end{eqnarray*}
Using \eqref{ProfLemP1Pap3},  we obtain that $$ \int_{0}^{1} \|  \rho(t_{n}-\sigma ) -g(-\sigma)  \| ds   \rightarrow 0 \; \text{ as } n\rightarrow \infty , $$
independently of $ t $ and $\varepsilon$. Therefore 
$$ x_{ \rho}^{\varepsilon}(t_{n})\rightarrow  y_{ \rho}^{\varepsilon} \; \text{ as } n\rightarrow \infty. $$
Consequently, $ x_{ \rho}^{\varepsilon} $ has a relatively compact range on $ [t_{0}, \infty) .$ $\ \blacksquare$
\end{proof}
\begin{lemma}\label{Propostion12}
Let $ f: \r \times X \longrightarrow X $ satisfy \textbf{(H4)} with $ \phi, \ \psi \in AAS^{1}(\r,X) $. Assume that \textbf{(H1)-(H2)} hold. If equation \eqref{Eq1SubStep} has at least a bounded mild solution
$x_{0}: [t_{0} ,+\infty) \longrightarrow X$, then $ x_{0} $ has a relatively compact range.  
\end{lemma}
\begin{proof} 
Let $ 0<\varepsilon <1$. Define $$ x_{0}(t)=T(\varepsilon)x_{0}(t-\varepsilon) +\int_{t-\varepsilon}^{t}T(t-\sigma)f(\sigma,x_{0}(\sigma))\ d\sigma
\qquad\text{ for }t> t_{0}+1.   $$
Let $ t>t_{0}+1 $. Then, we have
\begin{eqnarray}
T(\varepsilon)x_{0}(t-\varepsilon) \in T(\varepsilon)  \overline{B}(0,M_{0}) \label{CompacityX01}
\end{eqnarray}
where $ M_{0}:=\sup_{t\geq t_{0}}\parallel x_{0}(t)\parallel<+\infty  $ and $ \overline{B}(0,M_{0}) $ is the closed ball of center $ 0 $ and radius $ M_{0}$.\\
It suffices to prove that the measure of noncompactness $\alpha$ of $ \displaystyle \int_{t-\varepsilon}^{t}T(t-\sigma)f(\sigma,x_{0}(\sigma))\ d\sigma $ equals zero. In fact, by assumption on $f$, we obtain that
$$   \int_{t-\varepsilon}^{t}T(t-\sigma)f(\sigma,x_{0}(\sigma))\ d\sigma= \int_{t-\varepsilon}^{t}T(t-\sigma) \phi(\sigma) g(x_{0}(\sigma))\ d\sigma +\int_{t-\varepsilon}^{t}T(t-\sigma) \psi(\sigma) \ d\sigma  $$
\begin{eqnarray*}
\|\int_{t-\varepsilon}^{t}T(t-\sigma) \phi(\sigma) g(x_{0}(\sigma))\ d\sigma \|  \leq  Me^{|\omega | \varepsilon} | \phi|_{\infty}  \sup_{x\in  \overline{B}(0,M_{0}) }\| g(x)\| \ \varepsilon :=\gamma(\varepsilon) .  
\end{eqnarray*}
Hence, $$ \int_{t-\varepsilon}^{t}T(t-\sigma)f(\sigma ,x_{0}(\sigma))\ d\sigma \in \overline{B}(0, \gamma(\varepsilon) )+ \lbrace x^{\varepsilon}_{\psi}(t): \ t> t_{0}+1 \rbrace .$$
By Lemma \ref{Lemma1CaseP1} and \textbf{(H2)}, it follows that  $$ \alpha (\lbrace x_{0}(t):\; t\geq t_{0}\rbrace  ) \leq 2 \gamma( \varepsilon).$$
Since $\gamma( \varepsilon) \rightarrow 0$ as $ \varepsilon \rightarrow 0 $. We claim that 
$$ \alpha (\lbrace x_{0}(t):\; t\geq t_{0}\rbrace  ) =0.$$ 
\begin{flushright}
$\ \blacksquare$
\end{flushright}

\end{proof}
In general when $f$ does not satisfy \textbf{(H4)}, Lemma \ref{Propostion12} cannot hold. However, under the additional assumption \textbf{(H5)} below, we have the following main result which is of independent interest.\\

\textbf{(H5)} For all bounded subset $B \subset X$, the set $\Gamma := \lbrace f(t,x): \; (t,x) \in \r \times X \rbrace$ is relatively compact in $X$.\\
\begin{proposition}\label{Proposition H5}
Let $ f:\r \times X\longrightarrow X $ such that $ f(\cdot, x) \in BS^{1}(\r,X) $ for each $x\in X$. Assume that \textbf{(H1)}-\textbf{(H2)} and \textbf{(H5)} are satisfied. If equation \eqref{Eq1SubStep} has at least a bounded mild solution $x_{0}: [t_{0} ,+\infty) \longrightarrow X$, then, its range is relatively compact in $X$.
\end{proposition}
\begin{proof} 
Let $ 0<\varepsilon <1$. The mild solution $ x_{0} $ satisfies the following:  $$ x_{0}(t)=T(\varepsilon)x_{0}(t-\varepsilon) +\int_{t-\varepsilon}^{t}T(t-\sigma)f(\sigma,x_{0}(\sigma))\ d\sigma \qquad\text{ for }t \geq  t_{0}.   $$
Let $ \displaystyle x^{\varepsilon}_{0}(t):=\int_{t-\varepsilon }^{t} T(t-\sigma) f(\sigma, x_{0}(\sigma)) \ d\sigma . $ As in the proof of Lemma \ref{Propostion12}, it remins to prove that, the set $$\displaystyle  \lbrace x^{\varepsilon}_{0}(t): \; t >  t_{0}+1 \rbrace $$ is relatively compact in $X$. In fact, let $ (t'_{n})_{n} \subset \r$ and for each  $ t\in \r $, we define the sequence  $ (f(t'_{n}-t,x_{0}(t'_{n}-t)))_{n}$. Using \textbf{(H5)}, there exist as subsequence $  (t_{n})_{n} \subset (t'_{n})_{n} $ and $g: \r\longrightarrow X$ measurable such that for each $ t\in \r $ 
\begin{eqnarray}
f(t_{n}-t,x_{0}(t_{n}-t)) \rightarrow g(t) \; \text{as} \; n\rightarrow \infty. \label{CompCondProp1}
\end{eqnarray}
Given $ \displaystyle y^{\varepsilon}_{0}:=\int_{0 }^{\varepsilon} T(\sigma) g(\sigma) \ d\sigma . $ Then
 \begin{eqnarray*}
\| x^{\varepsilon}_{0}(t_{n})- y^{\varepsilon}_{0}\| & = &\| \int_{t_{n}}^{t_{n}-\varepsilon} T(t_{n}-s)  f(s,x_{0}(s))  ds- \int_{0 }^{\varepsilon} T(\sigma) g(\sigma) \ d\sigma \| \\ & \leq & M e^{|\omega | \varepsilon} \int_{0}^{\varepsilon} \|  f(t_{n}-\sigma,x_{0}(t_{n}-\sigma))  -g( \sigma)\|  d \sigma .
\end{eqnarray*}
Hence, using \eqref{CompCondProp1}, it follows in view of dominated convergence Theorem that   $$ \int_{0}^{\varepsilon} \|  f(t_{n}-\sigma,x_{0}(t_{n}-\sigma))  -g( \sigma)\|  d \sigma  \rightarrow 0 \, \text{as} \; n \rightarrow \infty ,$$ which implies that $$ \| x^{\varepsilon}_{0}(t_{n})- y^{\varepsilon}_{0}\|  \rightarrow 0  \, \text{as} \; n \rightarrow \infty . $$
Therefore,  $$ \alpha( \lbrace x^{\varepsilon}_{0}(t): \; t >  t_{0}+1 \rbrace ) . $$  
Consequently, using hypothesis \textbf{(H2)}, we obtain that
$$ \alpha (\lbrace x_{0}(t):\; t\geq t_{0}\rbrace )=0.$$ 
\begin{flushright}
$\ \blacksquare$
\end{flushright}
\end{proof} 
\subsection{Almost automorphic solutions minimizing a subvariant functional}\label{section32Pep3}
In the case $p\in (1,\infty)$, we have the following main Theorem.
\begin{theorem} \label{Main Theorem 1}
Let $ f  \in AAS^{p}U(\mathbb{R}\times X, X)$ and $K$ be a compact subset of $X$.  Assume that \textbf{(H1)-(H3)} are satisfied. Moreover, if equation \eqref{Eq1SubStep} has at least a mild solution $x_{0}$ defined and bounded on $[t_{0},+\infty)$. Then, the following hold:\\

\textbf{(i)}  If $\lambda_{K}$ is a subvariant functional associated to the compact set $K$, then equation \eqref{Eq1SubStep} has at least a minimal $K$-valued solution $x$. \\ 

\textbf{(ii)} If equation\eqref{Eq1SubStep} has a unique minimal $K$-valued solution $x$, then $x$ is compact almost automorphic.
\end{theorem}
\begin{proof}
Let $K:=\displaystyle \overline{ \{x_{0}(t):\ t\geq t_{0}\}} $, by Lemma \ref{Lemma Relativ Comp p not 1}, the set $K$ is compact in $X$.\\
\textbf{(i)} Define $\displaystyle  \delta=\inf_{x\in\mathcal{F}_{K}}\lambda_{K}(x)$.  Using  Lemma \ref{Propostion1}, we claim that the set  $ \mathcal{F}_{K} $ is nonempty.\\
Hence, there exists a sequence $(x_{n})_{n} \subset \displaystyle  \mathcal{F}_{K}$, such that
\begin{eqnarray}
\lim_{n\rightarrow+\infty}\lambda_{K}(x_{n})=\delta . \label{SequInf}
\end{eqnarray}
Given $ U_{n}:= \lbrace x_{n} : \; n\in  \n \rbrace \subset C(\mathbb{R},X)$, by definition of $\displaystyle  \mathcal{F}_{K}$, the set $ U_{n}(t):=\lbrace x_{n}(t) , \; n\in  \n \rbrace  \subset K$, for each
$t\in\mathbb{R}$.\\  
Since $ f  \in AAS^{p}U(\mathbb{R}\times X, X)$, it follows by Theorem \ref{Lemma31}, that tha family $ U_{n} $ is equicontinuous. Hence, in view of Arzela-Ascoli Theorem, we claim that $U_{n}$ is a relatively compact subset of $C(\mathbb{R},X)$ endowed with the topology of compact convergence. Therefore,
there exists a subsequence of $(x_{n})_{n}$ (steel denoted by $(x_{n})_{n}$) such that
\begin{eqnarray}
x_{n}(t)\rightarrow x_{\ast}(t)\qquad\mathrm{as}\quad n\rightarrow
+\infty, \label{ConvSubseqAsArz}
\end{eqnarray}
uniformly on each compact subset of $\mathbb{R}$. Obviously, $x_{\ast}\in
C(\mathbb{R},X)$ with range including in $ K $ (i.e., $x_{\ast}\in
C_{K}\mathbb{R},X)$). \\
Since, for every $n \in \n $, $x_{n}$ is a mild solution on $\mathbb{R}$ of equation \eqref{Eq1SubStep}, we claim (using the same proof as in Lemma \ref{Propostion1}) that 
$$
x_{\ast}(t)=T(t-s)x_{\ast}(s)+\int_{s}^{t}T(t-\sigma)f(\sigma,x_{\ast}%
(\sigma))\ d\sigma\qquad\text{ for }t\geq s.
$$
Thus $x_{\ast}$ is a mild solution on $\mathbb{R}$ of equation \eqref{Eq1SubStep}. This implies that $x_{\ast}\in\mathcal{F}_{K}$ and that
\begin{eqnarray}
\delta\leq\lambda_{K}(x_{\ast}). \label{InegSubFunFirst}
\end{eqnarray}
By \eqref{ConvSubseqAsArz} and Definition \ref{DefinitionSubvarFunct}, we obtain that
\begin{eqnarray}
\lambda_{K}(x_{\ast})\leq\liminf_{n\rightarrow+\infty}\lambda_{K} \label{AppDefinSubVArFun1}
(x_{n}).
\end{eqnarray}

From \eqref{ConvSubseqAsArz}, \eqref{InegSubFunFirst} and \eqref{AppDefinSubVArFun1}, we deduce that
\begin{eqnarray}
\lambda_{K}(x_{\ast})=\delta. \label{AppDefinSubVArFun2}
\end{eqnarray}
Consequently,  $x_{\ast}$ is a minimal $ K $-valued solution:
\begin{eqnarray}
\lambda_{K}(x_{\ast})=\inf_{x\in\mathcal{F}_{K}}\lambda_{K}(x). \label{AppDefinSubVArFun3}
\end{eqnarray}
This proves the existence of the minimal $K$-valued solution of equation \eqref{Eq1SubStep}.\\
\textbf{(ii)} Let $x_{\ast}$ be the unique minimal solution with relatively compact range
of equation \eqref{Eq1SubStep}. We show that $x_{\ast}$ is compact almost automorphic. In fact, let $(t'_{n})_{n}$ be a sequence of real numbers. Since $ f \in AAS^{p}U(\r \times X, X) $ it follows that 
there exist a subsequence $(t_{n})_{n} \subset (t'_{n})_{n} $ and a function $ g: \mathbb{R}\times X \longrightarrow X $ such that \eqref{DefAlmAuomUSubv} holds.\\
Define the set $ U_{n}=\lbrace x_{n}: \;  n\in \mathbb{N}  \rbrace \subset C(\r, X) $ where for each $ t\in \r $, $ x_{n}(t):=x_{\ast}(t+t_{n}) $.\\
 By assumption, for each $ t\in \r $, $ x_{n}(t) \in K $. Hence, for each $ t\in \r $ $$ U_{n}(t) \subset K .$$
Moreover, by Theorem \ref{Lemma31}, it holds in view of Arzela-Ascoli Theorem in  $ C(\r, X) $ endowed with topology of compact convergence that, there exist $ y_{\ast} \in C(\r,X)$ and a subsequence $ (t_{n}) $ such that $$ x_{n}(t) \rightarrow y_{\ast}(t) \quad \text{as}\; n\rightarrow \infty ,$$
uniformly on compact subsets of $ \r $. Therefore, $$ T(t-s)x_{n}(s) \rightarrow T(t-s) y_{\ast}(s) \quad \text{as}\; n\rightarrow \infty ,$$
uniformly on compact subsets of $ \r $. On the other hand, since $ x_{\ast} $ is a mild solution of equation \eqref{Eq1SubStep}, we have
\begin{eqnarray*}
x_{\ast}(t+t_{n})=T(t-s)x_{\ast}(s+t_{n})+\int_{s}^{t}T(t-\sigma)f(\sigma
+t_{n},x_{\ast}(\sigma+t_{n}))\ d\sigma, \quad t\geq s.
\end{eqnarray*}
Then, using the same reasoning as in Lemma \ref{Propostion1},  we claim that 
\begin{eqnarray*}
y_{\ast}(t)=T(t-s)y_{\ast}(s)+\int_{s}^{t}T(t-\sigma)f(\sigma
+t_{n},y_{\ast}(\sigma))\ d\sigma, \quad t\geq s.
\end{eqnarray*}
Consequently, $y_{\ast}$ is a mild solution on $\mathbb{R}$ with relatively compact range in $ K $ of equation
\begin{eqnarray}
x^{\prime}(t)=Ax(t)+g(t,x(t)), \quad t\in \mathbb{R}. \label{EqForG}
\end{eqnarray}
By definition of the subvariant $\lambda_{K}$, we obtain that
$\lambda_{K}(y_{\ast})\leq\lambda_{K}(x_{\ast})$. Then, from \eqref{AppDefinSubVArFun3}  we deduce that
\begin{eqnarray}
\lambda_{K}(y_{\ast})\leq\inf_{x\in\mathcal{F}_{K}}\lambda_{K}(x). \label{SubInfMildSol2}
\end{eqnarray}
Furthermore, since $ y_{\ast} $ is a mild solution with relatively compact range in $ K $ of equation \eqref{EqForG}, it follows in view of \eqref{DefAlmAuomUSubv} and by using the same reasoning as for $x_{\ast}$ by taking the sequence $ (-t_{n})_{n} $ instead of $ (t_{n})_{n} $, that there exists a mild solution $ z_{\ast} $ with relatively compact range in $ K $ of equation \eqref{Eq1SubStep}. Then, we have 
$$\lambda_{K}(z_{\ast})\leq\lambda_{K}(y_{\ast}). $$Therefore, by \eqref{SubInfMildSol2}, we obtain
\begin{eqnarray*}
\lambda_{K}(z_{\ast})\leq\inf_{x\in\mathcal{F}_{K}}\lambda_{K}(x). 
\end{eqnarray*}
Hence
$$\displaystyle  \lambda_{K}(z_{\ast})=\inf_{x\in\mathcal{F}_{K}}\lambda
_{K}(x) .$$ Therefore, $z_{\ast}$ is a minimal  solution with relatively compact range in $K$ of equation \eqref{Eq1SubStep}. By uniqueness we
deduce that $x_{\ast}=z_{\ast}$.  Consequently, $x_{\ast}$ is compact almost automorphic. $\ \blacksquare$
\end{proof}
From Theorem \ref{Main Theorem 1}, we deduce the following result. 
\begin{corollary}\label{Corollary1}
Let $ f  \in AAS^{p}U(\mathbb{R}\times X, X)$. Assume that \textbf{(H1)-(H3)} are satisfied. If
equation \eqref{Eq1SubStep} has a unique bounded mild solution $ x$ on $\R$, then $x$ is compact almost automorphic.
\end{corollary}
\begin{proof}
By assumptions, all hypotheses of Theorem \ref{Main Theorem 1} are satisfied and in particular, the set $K:=\displaystyle \overline{ \{x(t):\ t\geq 0\}} $ is compact. Hence, by taking $ \lambda_{K}\equiv 1 $, we conclude the result in view of Theorem \ref{Main Theorem 1}-\textbf{(ii)}. $\ \blacksquare$
\end{proof}
For $p=1$, Theorem  \ref{Main Theorem 1} and Corollary \ref{Corollary1} become :
\begin{theorem} \label{Main Theorem 2}
Let  $K$ be a compact subset of $X$.  Assume that \textbf{(H1)}  \textbf{(H2)} and  \textbf{(H4)} are satisfied. Moreover, if equation \eqref{Eq1SubStep} has at least a mild solution $x_{0}$ defined and bounded on $[t_{0},+\infty)$. Then, the following hold:\\
\textbf{(i)}  If $\lambda_{K}$ is a subvariant functional associated to the compact set $K$, then equation \eqref{Eq1SubStep} has at least a minimal $K$-valued solution $x$. \\ 
\textbf{(ii)} If equation\eqref{Eq1SubStep} has a unique minimal $K$-valued solution $x$, then $x$ is compact almost automorphic.
\end{theorem}
\begin{proof}
We argue as the same as in the proof of Theorem \ref{Main Theorem 1}, using Remarks  \ref{Remark31}, \ref{RemarkContSpG3} and Lemma \ref{Propostion12} in place of Lemma \ref{Lemma Relativ Comp p not 1}. $\ \blacksquare$
\end{proof}
Hence, we have the following Corollary.
\begin{corollary}\label{Corollary2}
Let  \textbf{(H1)},  \textbf{(H2)} and  \textbf{(H4)} be satisfied. If equation \eqref{Eq1SubStep} has a unique bounded mild solution $ x$, then $x$ is compact almost automorphic.
\end{corollary}
\section{Application}\label{section4}
Consider the following nonautonomous reaction-diffusion problem :
\begin{equation}
\left\{
\begin{array}
[c]{l}%
\dfrac{\partial}{\partial t}v(t,\xi)=\displaystyle \sum_{k=1}^{n}\dfrac{\partial^{2}}{\partial \xi_{i}^{2}%
}v(t,\xi)+g(v(t,\xi))+h\left(  t,\xi\right)  \text{ for }t\in\mathbb{R}\text{ and
}\xi\in\Omega ,\\
\text{ \ \ \ \ \ \ \ \ \ \ \ \ \ \ \ \ \ \ \ \ \ }\\
v(t,\xi)=0\; \text{ for }t\in\mathbb{R}, \ \xi\in \partial \Omega ,
\end{array}
\right.  \label{AppPap3}
\end{equation}
where $ \Omega \subset \r^{n}, \ n\geq 1$ is bounded and open with smooth boundary $ \partial \Omega $, $g:\mathbb{R}\rightarrow\mathbb{R}$ and $h:\mathbb{R}\times\Omega \rightarrow\mathbb{R}$.\\
Given the Banach space $X=C_{0}\left(  \overline{\Omega}  \right) := \lbrace x \in C(\overline{\Omega} , \r ): \, x \vert_{\partial \Omega }=0\rbrace $ 
endowed with the uniform norm topology. Let
\[
h(t,\xi)=\sin\left(  \frac{1}{2+cost+cos\sqrt{2}t}\right)  h_{0}(\xi) + a(t)\text{
for }t\in\mathbb{R}\text{ and }\xi\in \Omega
\]
where $ h_{0}\in X $ and $ \displaystyle a(t)= \sum_{n\geq 1} \beta_{n}(t)$ such that for every $ n \geq 1 $ $$ \beta_{n}(t)= \displaystyle \sum_{i \in P_{n}} H(n^{2}(t-i)),$$
with $ P_{n}=3^{n}(2\mathbb{Z}+1) $ and $ H \in C_{0}^{\infty}(\mathbb{R},\mathbb{R}) $ with support in $ (\frac{-1}{2},\frac{1}{2}) $ such that 
$$  H \geq 0 , \quad H(0)=1 \quad \text{and} \quad \int_{\frac{-1}{2}}^{\frac{1}{2}} H(s) ds =1 .  $$
\begin{remark}\cite{Tar}\label{Tar}
The function $a \in  C^{\infty}(\r,\r) $ but $a \notin  AA(\r,\r) $ since it is not bounded on $\r$.
However $ a\in  AAS^{1}(\r ,\r).$
\end{remark}

Define the operator $A:D(A)\subset
X\rightarrow X$ by
\[
\left\{
\begin{array}
[c]{l}%
D(A)=\left\{  x\in X \cap H_{0}%
^{1}\left(  \Omega \right)  :\ \displaystyle \sum_{k=1}^{n}\dfrac{\partial^{2}}{\partial \xi_{i}^{2}%
}x:= \Delta x \in X  \right\}  ,\\
Ax=\Delta x.
\end{array}
\right.
\]
It is well known that $(A,D(A))$ generates a compact $C_{0}$-semigroup $(T(t))_{t\geq0}$ on $X$ such that 
\begin{equation}
\parallel T(t)\parallel\leq M e^{-\lambda_{1} t}\qquad\text{ for }%
t\geq0 \label{Z},
\end{equation}
where $\lambda_{1}:= \min \{\lambda : \ \lambda \in \sigma(- A) \} > 0$, see \cite{Haraux} for more details.\\
 Consequently, hypotheses \textbf{(H1)} and  \textbf{(H2)} are satisfied. Moreover,  we assume that $g$  is locally Lipschitzian such that 
\begin{equation}
g(0)=0 \; \text{ and }  \; \limsup_{\mid r\mid\rightarrow+\infty}\dfrac{g(r)}{r}<\lambda_{1}. \label{Condition1 on g}
\end{equation}
\begin{remark}\label{Remark1 Appli}
Condition \eqref{Condition1 on g} implies: there exists $\bar{M}\geq 0$ such that $rg(r)\leq Cr^{2}$ for
$\mid r\mid\geq \bar{M}$ with $C<\lambda_{1} .$
\end{remark}
Consider $G:X\rightarrow X$ and $H:\mathbb{R}\rightarrow X$ defined by 
\begin{equation*}
G(x)(\xi)=g(x(\xi)),\qquad x\in X, \; \xi\in \Omega
\end{equation*}
and
\begin{equation*}
H(t)(\xi)=h(t,\xi),\qquad x\in X, \; \xi\in \Omega .
\end{equation*}
Hence, we denote by $f:\mathbb{R\times}X\rightarrow X$ the function defined by
\begin{equation}
f(t,x)=G(x)+H(t),\qquad t\in\mathbb{R}, \; x\in X. \label{f}
\end{equation}
Then, equation \eqref{AppPap3} is equivalent to the following abstract evolution equation:
\begin{equation}
x^{\prime}(t)=Ax(t)+f(t,x(t))\qquad\text{ for }t\in\mathbb{R}. \label{aa}
\end{equation}
The corresponding initial value problem is the following:
\begin{equation}
\left\{
\begin{array}
[c]{l}%
x^{\prime}(t)=Ax(t)+f(t,x(t))\qquad\text{ for }t\geq0\\
x(0)=x_{0},
\end{array}
\right.  \label{SemCPrb}%
\end{equation}
Now, under the above assumptions we prove the existence of a unique bounded global solution to equation \eqref{SemCPrb} under the weak assumption that $ H $ (and hence $f$) is not bounded on $ t $, see Remark \ref{Tar}.
\begin{theorem}\label{Theorem 4.1}
For each $ x_{0} \in X $ there exist $ t(x_0 ) >0 $ and a unique  maximal mild solution $x(\cdot ,x_{0}) $ of equation \eqref{SemCPrb} defined on $ [0, t(x_0 ) )$  such that 
\begin{equation}
t(x_0 )=+\infty \quad \text{or} \quad  \limsup_{t\rightarrow t(x_0 ) } \| x(t,x_0 )\| = +\infty .
\end{equation} \label{ETFFormula}
\end{theorem}
\begin{proof}
Fix $ R_{0} >0 $ and take $ t_{0}:= \log(2M\dfrac{(L_{R_{0}} +\lambda_{1}\tilde{M}+\lambda_{1})}{\lambda_{1}})\lambda_{1}^{-1} >0 $ where  $ \tilde{M}:= e^{[t_{0}]+1}( [t_{0}]+2) \| f(\cdot,0)\|_{BS^{1}} < \infty .$
$$  E= \{ u \in C([0,t_{0}],X):\ u(0)=x_{0} \; \text{and } \ \| u(t)\| \leq R_{0} \, \text{ for all }  t\in [0,t_{0}]  \} \subset C([0,t_{0}],X).$$
The set $ E $ equipped with the uniform norm topology is complete since it is closed in the Banach space $(C([0,t_{0}],X), \| \cdot\|_{\infty})$. Define the functional $ \Phi : E \longrightarrow C([0,t_{0}],X) $ by 
$$  \Phi(u)(t)=  T(t)x_{0}+ \int_{0}^{t} T(t-s)f(s,u(s))ds \quad \text{for } \, t\in [0,t_{0}].$$
Let $ u\in E $, we show that $ \Phi(u) \in E $. In fact, 
\begin{eqnarray*}
\|\Phi(u)(t)\| & \leq & Me^{-\lambda_{1} t_{0}}\|x_0 \| +  M \int_{0}^{t} e^{-\lambda_{1}(t-s)}\| f(s,u(s)) \| ds \\ &\leq & M e^{-\lambda_{1} t_{0}} R_{0}+  M \int_{0}^{t} e^{-\lambda_{1}(t-s)}\left[ \| f(s,u(s)) -f(s,0)\|+ \|f(s,0)\|\right]  ds 
\\ &\leq & M e^{-\lambda_{1} t_{0}} R_{0}+  M e^{-\lambda_{1} t_{0}}\int_{0}^{t} e^{\lambda_{1} s}L_{R_{0}} \| u(s)\| ds+ Me^{-\lambda_{1} t_{0}} \sum_{k=0}^{[t_{0}]+1} \int_{k}^{k+1} e^{\lambda_{1} s} \|f(s,0)\| ds 
\\ &\leq & e^{-\lambda_{1} t_{0}}  M\left(  1+  \dfrac{L_{R_{0}}}{\lambda_{1}}+ \tilde{M}\right)R_{0} =\dfrac{R_{0}}{2} \leq R_{0}  \quad \text{ for all } t\in \r .
\end{eqnarray*}
Then $ \Phi(E) \subset E $. Now, let $u, v \in E$, then we have
\begin{eqnarray*}
\|\Phi(u)(t)-\Phi(v)(t)\|& \leq &  M \int_{0}^{t} e^{-\lambda_{1}(t-s)}\| f(s,u(s)) -f(s,v(s))\| ds \\ &\leq &e^{-\lambda_{1} t_{0}} M \dfrac{L_{R_{0}} }{\lambda_{1}} \\ &\leq & e^{-\lambda_{1} t_{0}}  M\left(    \dfrac{L_{R_{0}}}{\lambda_{1}}+ \tilde{M} +1 \right)\| u-v\|_{\infty}  = \dfrac{\| u-v\|_{\infty} }{2}  \quad \text{ for all } t\in \r . 
\end{eqnarray*}
Hence  $$  \|\Phi(u)-\Phi(v)\|_{\infty}  \leq \dfrac{\| u-v\|_{\infty} }{2} $$which proves that $ \Phi$ is a contraction map. By Banach fixed point Theorem, we deduce the existence of a unique mild solution $x_{0}$  to equation \eqref{SemCPrb} defined on $ [0,t_{0}] $.
Let $ t_{\varepsilon} > t_{0} $ be well chosen and consider equation \eqref{SemCPrb} for $ t\in [t_{0},t_{\varepsilon}] $ with initial data $ x_{t_{0}}=x_{0}(t_{0}) $. Arguing as above, we can prove the existence of a unique mild solution $ x_{\varepsilon}:[t_{0},t_{\varepsilon}]\longrightarrow X  $ with $ x_{\varepsilon}(t_{0})=x_{t_{0}} $. Then, we define 
 \begin{equation*}
x=\left\{
\begin{array}
[c]{l}%
x_{0}\qquad\text{ in } [0,t_{0}] ,\\
x_{\varepsilon} \qquad\text{ in } [t_{0},t_{\varepsilon}],
\end{array}
\right. 
\end{equation*}
By construction $ x: [0,t_{\varepsilon}]\longrightarrow X  $ is continuous. Therefore $ x $ is the unique solution of equation \eqref{SemCPrb} in $  [0,t_{\varepsilon}] $. Hence, we adopt the same extension technique and we prove the existence of a unique maximal solution $ x: [0,t(x_{0}))\longrightarrow X  $ where 
$ t(x_{0}) := \sup \lbrace t_{0}>0 : \text{ there exist } x \in C([0,t_{0}],X) \text{ solution to \eqref{SemCPrb}}   \rbrace$.
Now,  we show formula \eqref{ETFFormula}. By contradiction, assume that $  t(x_{0})< \infty $ and there exists $C\geq 0$  such that $\|x(t)\|\leq C  $ for all $ t\in [0,t(x_{0})) $. We show that $ x $ is uniformly continuous on $ [0,t(x_{0})) $. Let $ (t_{n})_{n}, (s_{n})_{n} \in [0,t(x_{0}))$ be two sequences such that $ \sigma_{n} =t_{n}-s_{n}\rightarrow 0$ as $ n\rightarrow \infty $.  Without loss of generality we assume that $ 0\leq \sigma_{n} \leq 1 $ for all $ n\geq 0 $. Hence, we have
\begin{eqnarray}
\| T(t_{n})x_{0}-T(s_{n})x_{0}\| \leq  M\|T(\sigma_{n})x_{0}-x_{0}\| \rightarrow 0 \; \text{as} \; n\rightarrow \infty \label{UnifCont1}
\end{eqnarray}
follows from the strong continuity of the semigroup $ (T(t))_{t\geq 0} $ on $X$. Moreover, we have
\begin{eqnarray*}
&& \int_{0}^{t_{n}} T(t_{n}-s) f(s,x(s))ds -\int_{0}^{s_{n}} T(s_{n}-s) f(s,x(s))ds \\ &=& \int_{0}^{\sigma_{n}} T(t_{n}-s) f(s,x(s))ds - \int_{0}^{s_n } T(s_n -s)\left[  f(s,x(s)) - f(s-\sigma_n , x(s -\sigma_n)\right] ds 
\end{eqnarray*}
Note that by local Lipschitzity of $ g $ it holds that  $ f $ is also locally Lipschitzian on $X$ and it is continuous on $ \in [0,t(x_{0})] $ which implies that 
$$  \sup_{(t,x) \in  [0,t(x_{0})] \times \bar{B}(0,C)} \|f(t,x) \|= \tilde{C} <\infty .$$
Hence,  $$ \| T(t_{n}-s) f(s,x(s))\|  \leq Me^{\lambda_{1} t(x_{0})} \tilde{C} \quad \text{ for all } s \in [0, t(x_{0}) ) .$$ 
Then, by dominated convergence Theorem we claim that
$  \displaystyle\int_{0}^{\sigma_{n}} T(t_{n}-s) f(s,x(s))ds  \rightarrow 0 $ as $ n\rightarrow \infty $. Furthermore,  
\begin{eqnarray*}
&&  \| \int_{0}^{s_n } T(s_n -s)\left[  f(s,x(s)) - f(s-\sigma_n , x(s -\sigma_n)) \right] ds \| \leq \\ &  &  M_{\lambda_{1}}\int_{0}^{t(x_{0})} \left[ \|  f(s,x(s)) - f(s, x(s -\sigma_n)) \|   +  \|  f(s,x(s-\sigma_n)) - f(s-\sigma_n , x(s -\sigma_n)) \|\right]  ds \leq \\ 
& & M_{\lambda_{1}} L_{C}\int_{0}^{t(x_{0})} \|  x(s) - x(s -\sigma_n) \| ds +M_{\lambda_{1}} \int_{0}^{t(x_{0}) } \|  f(s,x(s-\sigma_n)) - f(s-\sigma_n , x(s -\sigma_n)) \| ds
\end{eqnarray*}
where $ M_{\lambda_{1}}=M e^{\lambda_{1} t(x_{0})} $. By the continuity of $ x $, it follows using dominated convergence Theorem that \\
 $ \displaystyle \int_{0}^{t(x_{0})} \|  x(s) - x(s -\sigma_n) \| ds;  \int_{0}^{t(x_{0}) } \|  f(s,x(s-\sigma_n)) - f(s-\sigma_n , x(s -\sigma_n)) \| ds \rightarrow 0$ as $ n\rightarrow \infty $.
 Thus  \begin{eqnarray} \| \int_{0}^{t_{n}} T(t_{n}-s) f(s,x(s))ds -\int_{0}^{s_{n}} T(s_{n}-s) f(s,x(s))ds \| \rightarrow 0  \; \text{as} \; n\rightarrow \infty . \label{UnifCont2}
 \end{eqnarray}
 Consequently, by \eqref{UnifCont1} and \eqref{UnifCont2} we claim that 
$$  \|x(t_n ) - x(s_n )\|  \rightarrow 0  \; \text{as} \; n\rightarrow \infty .$$ 
Then, $ x $ is uniformly continuous on $  [0,t(x_{0})) $ which implies that $\displaystyle  \lim_{t\rightarrow t(x_{0})^{-}} x(t)= x_{t(x_{0})} \in  X .$
 Therefore, there exists $ \eta >0 $ such that $ y $ is solution to equation \eqref{SemCPrb} in $  [T({x_{0}}),T({x_{0}})+\eta ]$ with initial data $y(T({x_{0}}))=x_{t(x_{0})} $. Hence
 \begin{equation*}
z=\left\{
\begin{array}
[c]{l}%
x\qquad\text{ in } [0,t({x_{0}})] ,\\
y \qquad\text{ in } [t({x_{0}}),t({x_{0}})+\eta ],
\end{array}
\right. 
\end{equation*}
defines a solution of equation \eqref{SemCPrb} in $  [0,t({x_{0}})+\eta ]$ which contradicts the fact that $ x $ is maximal. 
Consequently, formula \eqref{ETFFormula} is satisfied. $\ \blacksquare$
\end{proof}
\begin{lemma}\cite[Theorem 8.3.7]{Haraux}\label{Lemma Bounded global solution}
Assume $ H=0 $. Then, for each $ x_{0} \in X $ the following hold \\
\textbf{(i)} Equation \eqref{SemCPrb} has a unique global solution $ u_{x}$ given by:
$$ u_{x}(t)= T(t)x_0  + \int_{0}^{t} T(t-s) G(x(s)) ds, \quad t\geq 0 .$$
\textbf{(ii)} The solution $ u_{x} $ is bounded on $ \r ^{+} $ i.e., $  \sup_{t\geq 0} \| u_{x}(t)\| < \infty . $
\end{lemma}
\begin{theorem}\label{Proposition Application}
Let $ p \in [1,\infty) $. Then for each $ x_{0} \in X $ equation \eqref{SemCPrb} has a unique global bounded mild solution $x$ given by 
$$ x(t)= T(t)x_0  + \int_{0}^{t} T(t-s) G(x(s)) ds+\int_{0}^{t} T(t-s) H(s) ds, \quad t\geq 0 .$$
\end{theorem}
\begin{proof}
Fix $ x_{0} \in X $ and let $ x $ be the unique maximal solution to \eqref{SemCPrb} defined on $[0,t(x_{0}))$. Moreover, using Theorem \ref{Theorem 4.1} $x$ satisfies formula \eqref{ETFFormula}.  Note that a mild solution of equation \eqref{SemCPrb} is given by $$ x(t)=u_{x}(t)+\int_{0}^{t} T(t-s) H(s) ds , \quad t\in [0,t(x_{0})). $$ First, we prove that $x$ is global. In fact, by Lemma \ref{Lemma Bounded global solution}-\textbf{(i)}, it holds that $ \sup_{t\in [0,t(x_{0}))} \| u_{x}(t)\| < \infty $, if it is not, $ t(x_{0}) < \infty  $ which contradicts the fact that $ u_{x} $ is global. Hence
\begin{eqnarray}
\nonumber \|x(t)\| &\leq & \| u_{x}(t) \| +  \int_{0}^{t} \| T(t-s) H(s) \| ds \\ \nonumber
 &\leq & \| u_{x}(t) \| +  M\int_{0}^{t} e^{-\lambda_{1}(t-s)} \| H(s) \| ds \\ \nonumber
 &\leq & \| u_{x}(t) \| +  M\sum_{k=0}^{[t]+1}\int_{k}^{k+1} e^{-\lambda_{1}(t-s)} \| H(s) \| ds \\ 
  &\leq & \sup_{t\in [0,t(x_{0}))} \| u_{x}(t) \| +  \dfrac{M e^{3\lambda_{1}}}{e^{\lambda_{1}}-1} \| H\|_{BS^{p}}, \quad t\in [0,t(x_{0})). \label{BoundSoluGlob}
\end{eqnarray}
Then $$ \sup_{t\in [0,t(x_{0}))}  \|x(t)\| < \infty. $$

Consequently, by \eqref{ETFFormula}, we claim that $ t(x_{0}))=\infty. $ Thus $ x $ is global. \\
To conclude, we show that $ x $ is bounded on $ \r ^+ $. Using Lemma \ref{Lemma Bounded global solution}-\textbf{(ii)} we have $$ \sup_{t\geq 0} \| u_{x} (t)\| < \infty.$$ Therefore, by  \eqref{BoundSoluGlob} we can deduce that $$ \sup_{t\geq 0} \| x(t) \| < \infty . $$
\begin{flushright}
 $\ \blacksquare$
 \end{flushright} 
\end{proof}
 In order to prove the existence of a compact almost automorphic solution for equation  \eqref{AppPap3}, we need the following results.\\
 \begin{lemma} \label{Lemma Application 2}
 The function $f$ satisfies Hypothesis \textbf{(H4)} with $\phi =1$ and $\psi =H$. That is $ f \in AAS^{1}U(\r \times X ,X)$. 
 \end{lemma}
 \begin{proof}
 Let $ R\geq 0 $ and $ B \subset X$ bounded by $R$. By \eqref{Condition1 on g}, we have
\begin{eqnarray*}
 \| G(x)\| &\leq & L_{R} \|x\|  + \| G(0) \| \\ &\leq &  L_{R}.R   \quad  \text{for all } x \in B .
 \end{eqnarray*}
 This proves the fact for $G$. To conclude, it suffices to prove that $ H \in AAS^1 (\r ,X) $. In fact, let $ (t'_{n} )_{n} \subset \r$ be a sequence. Since $    a  \in AAS^{1}(\r, \r) $ (see \cite{Tar}) and $ b: t\longmapsto \sin\left(  \frac{1}{2+cost+cos\sqrt{2}t}\right) \in AAS^{1}(\r, \r) $, it follows that there exist $ (t_{n} )_{n} \subset  (t'_{n} )_{n} $ and $ \tilde{a} , \tilde{b}  \in L^{\infty}(\r,\r) $ such that, for each $t\in \r $ 
 \begin{equation}
 \int_{t}^{t+1} |a (s+t_{n})  -\tilde{a}(s)| ds \rightarrow 0 , \  \int_{t}^{t+1} |b (s+t_{n})  -\tilde{b}(s)| ds \rightarrow 0 \text{ as } n\rightarrow \infty . \label{AAS1 for H 1} 
 \end{equation}
 and 
 \begin{equation}
 \int_{t}^{t+1} |\tilde{a} (s-t_{n})  -a(s)| ds \rightarrow 0 , \  \int_{t}^{t+1} |\tilde{b} (s-t_{n})  -b (s)| ds \rightarrow 0 \text{ as } n\rightarrow \infty . \label{AAS1 for H 2} 
 \end{equation}
 Define $ \tilde{H}(t)(\xi)= \tilde{b}(t)h_{0}(\xi) + \tilde{a}(t) $ for $ t\in \r $ and $ \xi \in \Omega $. Hence, by \eqref{AAS1 for H 1}, we claim that for each $t\in \r $ 
 \begin{eqnarray*}
 \int_{t}^{t+1} \|H (s+t_{n})  -\tilde{H}(s)\| ds =\| h_{0}\|   \int_{t}^{t+1} |a (s+t_{n})  -\tilde{a}(s)| ds  + \int_{t}^{t+1} |b (s+t_{n})  -\tilde{b}(s)| ds \rightarrow 0
 \end{eqnarray*}
 as $ n\rightarrow \infty. $ By the same way,  we prove that for each $t\in \r$
 \begin{eqnarray*}
 \int_{t}^{t+1} \|\tilde{H} (s-t_{n})  -H(s)\| ds  \rightarrow 0 \text{ as } n\rightarrow \infty.
 \end{eqnarray*} 
 By Remark \eqref{RemarkContSpG3}-\textbf{(a)},  we conclude that $ f \in AAS^{1}U(\r \times X ,X)$.
 \begin{flushright}
$ \blacksquare$
 \end{flushright}
 \end{proof}
 For the uniqueness we need the following additional assumption on $g$:\\
 \begin{equation}
\text{The function } r\longmapsto g(r)-r  \text{ is nonincreasing on } \mathbb{R}.
 \end{equation}

Consider the function $E:X\rightarrow\mathbb{R}$ defined by
\begin{equation}
E\left(  x\right)  =\frac{1}{2}\int_{\Omega} \mid x(\xi)\mid^{2}\ d\xi.
\label{E}%
\end{equation}
\begin{lemma}\cite{CieuEzz} \label{Lemma App3}
Let $u$ and $v$ be two mild solutions of equation (\ref{aa}).
Then the following hold. \\
 \textbf{(i)} The function
$t\longmapsto E(u(t)-v(t))$ is nonincreasing on $\mathbb{R}$. \\
\textbf{(ii)} If $t\longmapsto E(u(t)-v(t))$ is constant on
$\mathbb{R}$, then there exists $w_{0}\in X$ such that
\begin{equation}
u(t)-v(t)=w_{0},\qquad t\in\mathbb{R}, \label{E1}
\end{equation}
\begin{equation}
G(u(t))-G(v(t))=-\lambda_{1}w_{0},\qquad t\in\mathbb{R}.\label{7.1}%
\end{equation}
Moreover, for all $\theta\in\lbrack0,1]$, $\theta u+(1-\theta)v$ is also a
mild solution of equation (\ref{aa}). 
\end{lemma} 
\begin{theorem}\label{Theorem Application main}
Under the above assumptions. The following statements hold
for equation \eqref{aa}:\\
\textbf{(i)} Equation \eqref{aa} has a at least one mild solution $x_{1}$ which is compact almost automorphic . If $x_{2}$ is a mild solution which is
only almost automorphic, then there exists $w_{0} \in X$ such that
\begin{equation}
x_{2}(t)=x_{1}(t)+w_{0}\qquad\text{ for }t\in\mathbb{R} .\label{P6}%
\end{equation}
Consequently $x_{2}$ is compact almost automorphic. \\
\textbf{(ii)} If the function $r\rightarrow g(r)-r$ is strictly decreasing on $\mathbb{R}$, then the compact almost mild solution is unique
\end{theorem}
\begin{proof}
We use Theorem \ref{Main Theorem 2} To prove the existence of a compact almost automorphic solution for equation \eqref{aa}. By Lemma \ref{Lemma Application 2} Hypotheses \textbf{(H1)},  \textbf{(H2)} and  \textbf{(H4)} are satisfied. In addition, from Proposition \ref{Proposition Application}, equation  \eqref{SemCPrb} 
has at least a mild solution $x_{0}$ defined and bounded on $[0,+\infty)$.
Let $K:=\displaystyle  \overline{\{x_{0}(t)\ :\ t\geq0\}}$, by Lemma \ref{Propostion12}, $K$ is compact in $X$. Let $co(K)$ be the convex hull of the compact set
$K$. Then $co(K)$ is a compact convex subset
of $X$. It is clear that $ K \subset co(K) $.\\
Define the subvariant functional $\lambda
_{co(K)}: \mathcal{F}_{co(K)} \longrightarrow [0, \infty)  $ by $$\displaystyle  \lambda
_{co(K)}(x)=\sup_{t\in\mathbb{R}}E(x(t))$$ associated to the compact set $co(K)$. We prove that equation \eqref{aa} has at most a minimal $co(K)$-valued solution. Let $u$
and $v$ be two minimal $co(K)$-valued solutions of equation \eqref{aa}. Define
\begin{equation}
\inf_{x\in \mathcal{F}_{co(K)} }  \lambda
_{co(K)}(x) :=\delta=\sup_{t\in\mathbb{R}}E\left(  u(t)\right)  =\sup_{t\in\mathbb{R}%
}E(v(t)). \label{7.3}%
\end{equation}
 \textbf{Case I. } Assume that $E\left(  u(t)-v(t)\right)
)=c \in \r^{+}$ for all $t\in\mathbb{R}$.  By the convexity of  $ co(K)  $ and Lemma \ref{Lemma App3}, the set $\mathcal{F}_{co(K)}$ is convex too.  Then
$\displaystyle  {\frac{1}{2}}u+{\frac{1}{2}}v\in\mathcal{F}_{co(K)}$. Therefore
\begin{equation}
\delta\leq\sup_{t\in\mathbb{R}}E\left(  {\frac{1}{2}}u(t)+{\frac{1}{2}%
}v(t)\right)  . \label{7.4}%
\end{equation}
On the other hand, by parallelogram inequality, we have
\[
E\left(  {\frac{1}{2}}u(t)+{\frac{1}{2}}v(t)\right)  +E\left(  {\frac{1}{2}%
}u(t)-{\frac{1}{2}}v(t)\right)  ={\frac{1}{2}}E\left(  u(t)\right)
+{\frac{1}{2}}E\left(  v(t)\right)  ,
\]
Then
\begin{equation}
\sup_{t\in\mathbb{R}}E\left(  {\frac{1}{2}}u(t)+{\frac{1}{2}}v(t)\right)
+{\frac{1}{4}}c\leq{\frac{1}{2}}\sup_{t\in\mathbb{R}}E\left(  u(t)\right)
+{\frac{1}{2}}\sup_{t\in\mathbb{R}}E\left(  v(t)\right)  . \label{7.5}%
\end{equation}
Consequently, $$\displaystyle  \delta+{\frac{1}{4}%
}c\leq{\frac{1}{2}}\delta+{\frac{1}{2}}\delta$$ which implies that  $c\leq0$. Hence, by definition of $E$, we
obtain $u=v$. \\
\textbf{Case II}:  Let $(t_{n})_{n}$ be a sequence of real numbers such that $\displaystyle
\lim_{n\rightarrow+\infty}t_{n}=-\infty$. Since by Lemma \ref{Lemma Application 2}, $ f \in AAS^{1}U(\r \times X, X)  $. Then,  by applying Theorem \ref{Main Theorem 1} two times for $ u $ and $ v $ respectively, we claim that there exists a
subsequence of $(t_{n})_{n}$ such that
\begin{equation}
u(t+t_{n})\rightarrow u_{1}(t)\qquad\mathrm{as}\qquad n\rightarrow+\infty,
\label{7.8}%
\end{equation}%
\begin{equation}
u_{1}(t-t_{n})\rightarrow u_{2}(t)\qquad\mathrm{as}\qquad n\rightarrow+\infty,
\label{7.9}%
\end{equation}%
\begin{equation}
v(t+t_{n})\rightarrow v_{1}(t)\qquad\mathrm{as}\qquad n\rightarrow+\infty,
\label{7.10}%
\end{equation}%
\begin{equation}
v_{1}(t-t_{n})\rightarrow v_{2}(t)\qquad\mathrm{as}\qquad n\rightarrow+\infty,
\label{7.11}%
\end{equation}
uniformly on each compact subset of $\mathbb{R}$, where $u_{2}$ and $v_{2}$
are two minimal $co(K)$-valued solutions of equation (\ref{aa}). Using (\ref{7.8})-(\ref{7.11}), we claim that for each
$t\in\mathbb{R},$ we have
\begin{equation*}
\lim_{m\rightarrow+\infty}\lim_{n\rightarrow+\infty}E\left(  u(t+t_{n}%
-t_{m})-v(t+t_{n}-t_{m})\right)  =E\left(  u_{2}(t)-v_{2}(t)\right)  . \label{Almost Aut E}
\end{equation*}
Moreover, since by Lemma \ref{Lemma App3}, the function $t\rightarrow E(u(t)-v(t))$ is nonincreasing on $\mathbb{R}$  and $\displaystyle  \lim_{n\rightarrow+\infty
}t_{n}=-\infty$, it follow that for each $t\in\mathbb{R} $
\[
\lim_{n\rightarrow+\infty}E\left(  u(t+t_{n})-v(t+t_{n})\right)  =\sup
_{\tau\in\mathbb{R}}E\left(  u(\tau)-v(\tau)\right)  .
\]
Therefore, 
\begin{equation}
E\left(  u_{2}(t)-v_{2}(t)\right)  =\sup_{\tau\in\mathbb{R}}E\left(
u(\tau)-v(\tau)\right)  . \label{7.12}%
\end{equation}
By (\ref{7.12}) and \textbf{Case I}, we obtain that $u_{2}=v_{2}$. By definition of
$E$, we deduce that
\[
\sup_{\tau\in\mathbb{R}}E\left(  u(\tau)-v(\tau)\right)  =0,
\]
Consequently $u=v$ and then equation  \eqref{aa}  has at most one minimal $co(K)$-valued solution. \\
\textbf{(i)} Using Theorem \ref{Main Theorem 2}, we claim that
equation \eqref{aa} has at least a compact almost automorphic solution.
 Let $x_{1}$ be a mild solution which is compact almost
automorphic. If $x_{2}$ is mild solution which is almost automorphic, then by \eqref{Almost Aut E} and Lemma \ref{Lemma App3} respectively, the function $t\rightarrow E(x_{1}(t)-x_{2}(t))$ is almost automorphic and nonincreasing. Therefore it is constant on $\mathbb{R}$. Consequently, in view of Lemma \ref{Lemma App3}\textbf{-(ii)}, we obtain \eqref{P6}.\\
 \textbf{(ii)} If the function $r\rightarrow g(r)-r$ is strictly decreasing on $\r$, then the uniqueness of the compact almost
automorphic solution results directly from (\ref{E1}) and (\ref{7.1}). $\ \blacksquare$
\end{proof}

\Addresses
\Class
\end{document}